\documentclass{siamart171218}

\usepackage{amssymb}
\usepackage{amsfonts}
\usepackage{stmaryrd}
\usepackage{csquotes}
\usepackage{microtype}

\newsiamremark{remark}{Remark}
\newsiamremark{assumption}{Assumption}

\hypersetup{
  pdftitle={On Particles and Splines in Bounded Domains},
  pdfauthor={Matthias Kirchhart}
}

\newcommand{\TheTitle}{On Particles and Splines in Bounded Domains} 
\newcommand{\TheAuthors}{Matthias Kirchhart}
\headers{\TheTitle}{\TheAuthors}

\title{On Particles and Splines in Bounded Domains}
\author{\TheAuthors\footnotemark[1]\thanks{Center for Computational
Engineering Science, Mathematics Division. RWTH Aachen University, Germany.
(\email{kirchhart@mathcces.rwth-aachen.de})}}

\newcommand{\spdim}{\mathrm{D}}   % Spatial Dimension
\newcommand{\reals}{\mathbb{R}}
\newcommand{\wholespace}{\reals^\spdim}
\newcommand{\gdom}{\bigcirc} % General Domain placeholder
\newcommand{\qdom}{\square}  % Finite collection of cubes placeholder.

\newcommand{\pd}{\partial}
\newcommand{\ddx}[2]{\frac{{\mathrm d}{#1}}{{\mathrm d}{#2}}}

\newcommand{\pdx}[2]{\frac{\partial {#1}}{\partial {#2}}}
\newcommand{\idx}[4]{\int_{#1}^{#2} {#3}\,{\mathrm d}{#4}}

\DeclareMathOperator{\supp}{supp}
\DeclareMathOperator{\clos}{clos}
\DeclareMathOperator{\inte}{int}

\DeclareMathOperator{\meas}{meas}
\DeclareMathOperator{\cond}{cond}

\newcommand{\vv}[1]{\mathbf{#1}}        % Vectors
\newcommand{\gv}[1]{\boldsymbol{#1}}    % Vectors  using Greek symbols
\newcommand{\sm}[1]{\mathsf{#1}}        % 'Algebraic' matrix.
\newcommand{\sv}[1]{\mathsf{#1}}

\begin{document}
\maketitle
% REQUIRED

\begin{abstract}
We propose numerical schemes that enable the application of particle methods for
advection problems in general bounded domains. These schemes combine particle
fields with Cartesian tensor product splines and a fictitious domain approach.
Their implementation only requires a fitted mesh of the domain's boundary, and
not the domain itself, where an unfitted Cartesian grid is used. We establish
the stability and consistency of these schemes in $W^{s,p}$-norms, $s\in\reals$,
$1<p\leq\infty$.
\end{abstract}

% REQUIRED

% REQUIRED

\section{Introduction}
We begin by introducing a simple toy problem: let $\Omega\subset\wholespace$ be
an open, bounded Lipschitz domain and let $\vv{a}: \bar{\Omega}\times[0,T]\to
\wholespace$ denote a given, smooth velocity field. Moreover, let us for
simplicity assume that $\vv{a}$ satisfies $\vv{a}\cdot\vv{n} = 0$ on the boundary
$\pd\Omega$, such that we do not need to worry about boundary conditions. We are
then interested in solving the initial value problem for the transport equation,
i.\,e., given initial data $u_0:\Omega\to\reals$, find $u: \Omega\times[0,T)
\to\reals$ such that:
\begin{equation}\label{eqn:advequation}
\begin{split}
\left\lbrace
\begin{aligned}
\pdx{u}{t} + (\vv{a}\cdot\nabla)u &= 0            &    \text{in }&\Omega\times(0,T), \\
u(\vv{x},0)                       &= u_0(\vv{x})  &    \text{on }&\Omega.
\end{aligned}
\right.
\end{split}
\end{equation}

\subsection{Grid-based Schemes}
It is well-known that the discretization of this problem with conventional
grid-based schemes such as finite differences, volumes, or elements causes a lot
of problems when $\vv{a}$ is large: for explicit time-stepping schemes the
CFL-condition forces one to use tiny time-steps. For the spatial discretization,
on the other hand, a common approach to guarantee stability is upwinding.
But this comes at the cost of introducing significant amounts of spurious,
numerical viscosity: in a numerical solution with $\vv{a}\equiv\mathrm{const}$
an initial step function $u_0$ quickly turns into a \enquote{ridge} of ever
decreasing slope. This is the source of many of the difficulties experienced
in numerical simulations of turbulent flows and computational fluid dynamics
in general. In short, while there certainly are more advanced schemes, it is
fair to say that it is very hard to construct grid-based methods that are
accurate, stable, and efficient when applied to advection problems.

\subsection{Particle Methods}
Particle methods like Smoothed Particle Hydrodynamics~(SPH) or Vortex Methods~(VM)
pursue a quite different approach to tackle this problem. Here, the initial
data $u_0$ is approximated with a special quadrature rule $u_{0,h}$ called
\emph{particle field}. It consists of weights $U_i\in\reals$ and associated
nodes $\vv{x}_i\in\Omega$, $i=1,\ldots,N$, such that for arbitrary smooth
functions $\varphi$ one has:
\begin{equation}
\sum_{i=1}^{N}U_i\varphi(\vv{x}_i) \approx \idx{\Omega}{}{u_0\varphi}{\vv{x}}.
\end{equation}
Equivalently, $u_{0,h}$ may be interpreted as a functional: $u_{0,h}=\sum_{i=1}^{N}
U_i\delta_{\vv{x}_i}$, where $\delta_{\vv{x}_i}$ denotes the Dirac $\delta$-%
functional centered at $\vv{x}_i$. The reason for such an approximation is
as follows. Given such discretized initial data  $u_{0,h}\approx u_0$, it can
be shown that the problem~\eqref{eqn:advequation} is well-posed and that its
unique solution is given by moving the particles with the flow, i.\,e., by
modifying $\vv{x}_i$ over time according to:
\begin{equation}\label{eqn:particleode}
\ddx{\vv{x}_i}{t}(t) = \vv{a}\bigl(\vv{x}_i(t),t\bigr)\qquad i = 1,\ldots,N.
\end{equation}
The fact that \emph{this is the exact solution} means that apart from the
discretization of the initial data, no further error is introduced by the
spatial discretization over time. Moreover, for $\vv{a}\in L^\infty\bigl(
W^{n,\infty}(\Omega),[0,T]\bigr)$, $n\in\mathbb{N}$, one can show that the
advection equation is stable in the sense that for all $t\in [0,T]$ the
following holds:\footnote{Here and throughout this text the notation $a\lesssim
b$ will mean that there exists a constant $C>0$ independent of $a$, $b$, $h$,
and $\sigma$ such that $a\leq Cb$. The variables $h$ and $\sigma$ refer to
certain mesh sizes and will be made precise later.} 
\begin{equation}\label{eqn:stableadvection}
\Vert u(t)\Vert_{W^{s,p}(\Omega)}\lesssim
\Vert u_0\Vert_{W^{s,p}(\Omega)}, \qquad -n\leq s\leq n.
\end{equation}
In this clarity these facts seem to first have been established by Raviart
\cite{raviart1985} and Cottet~\cite{cottet1988} in the 1980s. In the context of
particle methods the Dirac $\delta$-functional has already been mentioned in
1957 in Appendix~II of Evans' and Harlow's work on the Particle-in-Cell method;
\cite{harlow1957b} particle methods themselves at least date back to the early
1930s and Rosenhead's vortex sheet computations.~\cite{rosenhead1931} In practice
the ODE system~\eqref{eqn:particleode} is solved numerically using, e.\,g., a
Runge--Kutta scheme and it can be shown that there is no time-step constraint
depending on the discretization to guarantee the stability of the method.
Simulations with billions of particles have been carried out,~\cite{yokota2013} 
and practice has shown that particle schemes have excellent conservation
properties and are virtually free of numerical viscosity. In short,
particle methods are ideally suited for advection problems.

\begin{figure}\centering
\includegraphics{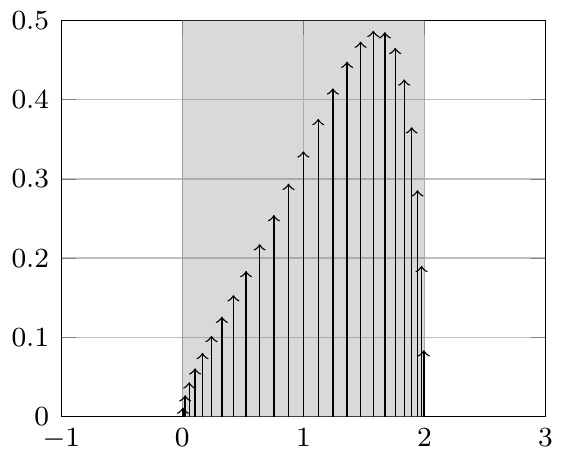}\hspace{3em}%
\includegraphics{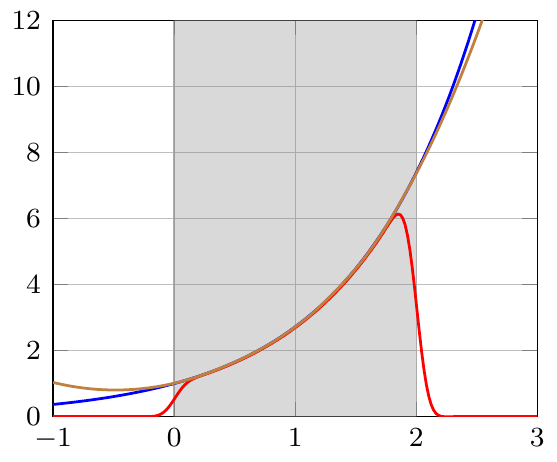}\hspace{3em}%
\caption{\label{fig:regularisation}Approximation of the exponential function
({\color{blue}blue}) on the interval $[0,2]$. On the left: a highly accurate
particle approximation. The particle weights, depicted by the arrows' heights,
usually do not correlate well with the local function values. On the right:
conventional smoothing of the particle field yields a globally smooth
approximation~({\color{red}red}) of the target function's non-smooth
zero-extension. This results in poor approximations near its discontinuities at
the boundaries. The stabilized $L^2$-projection~({\color{brown}brown})~\cite{kirchhart2017b}
yields an approximation of a smooth extension. It is not only accurate on the
entire interval but also extra\-polates well after its ends.}
\end{figure}

\subsection{Particle Methods in Bounded Domains}
A particle field $u_h$ can only be interpreted as a special quadrature rule; it
is important to understand that \emph{the $U_i$ are weights and not function
values.} In general, the $U_i$ do \emph{not} give a good picture of the local
values $u(\vv{x}_i)$ of the approximated function, much like the quadrature
weights from ordinary quadrature rules do not give a good picture of the number~1. 
This can for example clearly be seen on the left of \cref{fig:regularisation},
where a highly accurate particle approximation of the exponential function
on the interval $[0,2]$ is depicted. In reality, however, one is of course interested
in function values and a particle approximation is of little practical use. This
work therefore focuses on the following two questions:
\begin{enumerate}
\item Given a function $u$, how does one construct a particle approximation
$u_h\approx u$ and what error bounds does it fulfill? Here, $h$ denotes some
form of particle spacing and will be defined precisely later. This problem is
called \emph{particle initialization.}
\item Given a particle approximation $u_h\approx u$, how does one obtain a
function approximation $u_\sigma\approx u_h$, and what error-bounds does it
fulfill? Here, $\sigma$ denotes a smoothing length, which will also be defined
precisely later. This problem is called \emph{particle regularization.}
\end{enumerate}

While in the whole space case these questions are well understood, one of the
reasons why particle methods are so rarely used in engineering practice is the
difficulty to answer these questions in general bounded domains. In this work we
will develop and analyze schemes which aim two solve these problems. Our proposed
solutions will only require a mesh of the boundary $\pd\Omega$, not of the domain
$\Omega$ itself. Instead, a simple, unfitted Cartesian mesh is used for $\Omega$,
which can be obtained easily by a process known in computer graphics as \enquote{voxelization}.

The first problem can essentially be solved using quadrature rules. Especially
particle regularization, however, is not obvious in the presence of boundaries.
The most common approach to the regularization problem is to mollify the particle
field with a certain, radially symmetric \emph{blob-function} $\zeta_\sigma$:
$u_\sigma:=u_h\star\zeta_\sigma$, where $\sigma$ denotes the radius of the blob's
core.~\cite[Section~2.3]{cottet2000} These blobs are \enquote{unaware} of the
boundaries and yield poor approximations in their vicinity. In fact, this approach
yields globally smooth approximations of the zero-extension of $u$. Unless $u$
itself and its derivatives vanish on $\pd\Omega$, however, this extension is not
smooth and cannot be well approximated with a smooth function. This is depicted
on the right of \cref{fig:regularisation}. In Particle-in-Cell schemes one
uses interpolation formulas to obtain a grid-based approximation of the particle
field. In the vicinity of boundaries these formulas need to be specifically
adapted to the particular geometry at hand and cannot be used for arbitrary
domains. Recently, however, Marichal, Chatelain, and Winckelmans~\cite{marichal2016}
introduced a promising interpolation scheme for general boundaries, but a rigorous
error analysis seems unavailable at this time. They also give a review of some
other previous approaches and come to the conclusion that \enquote{None of the
schemes above truly succeeds in the generation of accurate particle -- or grid --
values around boundaries of arbitrary geometry.}

Recently, we proposed another approach to the regularization problem, which is
based on the $L^2$-projection and allows a rigorous analysis.~\cite{kirchhart2017b}
First, $C^\infty$-smooth finite-element spaces $V_\sigma$ on simple uniform
Cartesian grids are created, where $\sigma$ denotes the length of the cells. Then
a fictitious domain approach is employed and one searches the $L^2$-projection of
$u$ onto $V_\sigma$. In other words one looks for $u_\sigma\in V_\sigma$ such that
\begin{equation}
\idx{\Omega}{}{u_\sigma v_\sigma}{\vv{x}} = \idx{\Omega}{}{uv_\sigma}{\vv{x}}
\qquad
\forall v_\sigma\in V_\sigma.
\end{equation}
If one is only given a particle approximation $u_h\approx u$, the integral on
the right is replaced by $\sum_{i=1}^{N}U_iv_\sigma(\vv{x}_i)$. The addition
of a high-order stabilization term then ensures accuracy and stability of
the method independent of the position of the boundary $\pd\Omega$ relative
to the Cartesian grid. It was established that the resulting $u_\sigma$ then
approximates a smooth extension of $u$ and is optimal in a certain sense. The
result of this approach corresponds to the brown line on the right of
\cref{fig:regularisation}, and this figure clearly highlights its accuracy at
the boundaries and even beyond.

\subsection{Novelty and Main Result}
In this work we are going to build on and extend the results from our previous
work in several ways. The spaces $V_\sigma$ from~\cite{kirchhart2017b} have the
disadvantage that  an explicit representation of their basis functions is unavailable.
As a first step we are therefore replacing these spaces with Cartesian tensor product
splines. Secondly, we extend our error analysis to general $W^{s,p}$-spaces, with
$s\in\mathbb{R}$, $1\leq p\leq\infty$. It turns out that splines and particles seem
to ideally complement each other and it is also possible to solve the problem of
initialization. The main result of this work is summarized in \cref{thm:main}. The
obtained error bounds closely mirror those given by Raviart~\cite{raviart1985} for
the blob-based regularization in the whole-space case. At the same time, our method
works for general bounded domains and is faster: evaluation of the obtained regularized
field costs $\mathcal{O}(1)$ operations, compared to $\mathcal{O}(N_{\text{particles}})$
for the blob-based approach.

\section{Spaces of Functions and Functionals}\label{sec:spaces}
In this section we will introduce the function spaces and recall some important
results that our analysis will make use of. Throughout this text we will assume
that the domain of interest $\Omega\subset\wholespace$ is an open, bounded set
that satisfies the strong local Lipschitz condition; for short, a Lipschitz domain.
This assumption will in particular allow us to make use of the Sobolev embeddings
as well as the Stein extension theorem. The symbol $\gdom\subset\wholespace$ will
be used as a placeholder for any bounded Lipschitz domain.

\subsection{Sobolev Spaces of Integer Order}
As usual, for $n\in\mathbb{N}$, $1\leq p\leq\infty$, the Sobolev spaces $W^{n,p}(\gdom)$
are given by:
\begin{equation}
\begin{split}
W^{n,p}(\gdom) & := \bigl\lbrace f: \gdom\to\reals\,\big|\, \partial^\alpha f\in L^p(\gdom)
\ \forall 0\leq |\alpha| \leq n\bigr\rbrace, \\
\Vert\cdot\Vert_{W^{n,p}(\gdom)} &:= \biggl(\sum_{k=0}^{n}|\cdot|_{W^{k,p}(\gdom)}^p\biggr)^{1/p}, \\
|\cdot|_{W^{k,p}(\gdom)} &:= \biggl(\sum_{|\alpha|=k}\Vert\pd^\alpha(\cdot)\Vert_{L^p(\gdom)}^p\biggr)^{1/p},
\end{split}
\end{equation} 
where $W^{0,p}(\gdom) := L^p(\gdom)$, $\alpha\in\mathbb{N}_0^\spdim$ denotes
a multi-index, $\pd^\alpha$ the weak derivative, and the usual modifications for
$p=\infty$:
\begin{equation}
\begin{split}
\Vert\cdot\Vert_{W^{n,\infty}(\gdom)} := \max_{k=0,\ldots,n}|\cdot|_{W^{k,\infty}(\gdom)}, \\
|\cdot|_{W^{k,\infty}(\gdom)} := \max_{|\alpha|=k}\Vert\pd^\alpha(\cdot)\Vert_{L^\infty(\gdom)}. \\
\end{split}
\end{equation}

Whenever the index $1\leq p\leq\infty$ appears, we define $q$ as its Hölder
conjugate such that $\frac{1}{p} + \frac{1}{q} = 1$. For $n\neq 0$ we
\emph{define} $W^{-n,p}(\gdom) := \bigl(W^{n,q}(\gdom)\bigr)^{\prime}$ to
be the normed dual of $W^{n,q}(\gdom)$, following the convention of, e.\,g.,
Brenner and Scott,~\cite{brenner2008} \emph{but opposed to the convention}
$W^{-n,p}(\gdom) = \bigl(W^{n,q}_0(\gdom)\bigr)^{\prime}$ of Adams and
Fournier.~\cite{adams2003} We define the norm $\Vert\cdot\Vert_{W^{-n,p}(\gdom)}$
as usual, denoting the duality paring by $\langle\cdot,\cdot\rangle$:
\begin{equation}
\Vert\cdot\Vert_{W^{-n,p}(\gdom)} := \sup_{v\in W^{n,q}(\gdom)}
\frac{\langle\cdot, v\rangle}{\Vert v\Vert_{W^{n,q}(\gdom)}}.
\end{equation}

\subsection{Sobolev and Besov Spaces of Fractional Order}
We will later introduce spline spaces of approximation order $n\in\mathbb{N}$.
In terms of integer order Sobolev regularity, these splines however only lie
in $W^{n-1,p}$, which in the end would only allow us to prove suboptimal
results. We thus introduce intermediate spaces of fractional order, in terms
of which the splines possess the necessary amount of regularity.

We define intermediate spaces of fractional order using the \enquote{real}
interpolation method.~\cite[Chapter~7]{adams2003} In particular, for $0<
\theta<1$, $1\leq p,p'\leq\infty$ we define the Besov spaces as $B_{p'}^{\theta
n,p}(\gdom) := \bigl[L^p(\gdom),W^{n,p}(\gdom)\bigr]_{\theta,p'}$. Here $s:=
\theta n\in\reals_{+}$ measures the smoothness and $1\leq p\leq\infty$ denotes
the underlying $L^p(\gdom)$-space. Varying the secondary index $p'$ for fixed
values of $s$ and $p$ only results in miniscule changes; bigger values of $p'$
result in slightly larger spaces: $B_{p_1^\prime}^{s,p}(\gdom)\hookrightarrow
B_{p_2^\prime}^{s,p}(\gdom)$, $1\leq p_1^\prime \leq p_2^\prime\leq\infty$. On
the other hand, for every $r>s$ and $1\leq p'\leq\infty$ we have $B_{p'}^{r,p}%
(\gdom)\hookrightarrow B_1^{s,p}(\gdom)$. For $p\neq\infty$ this definition of
Besov spaces is equivalent to the one using appropriate moduli of smoothness.
(For $0<p\leq 1$ this has been established by De{\kern -.8pt}Vore and Sharpley.
\cite[Theorem~6.3]{devore1993} For $1<p<\infty$ a proof can be found in Adams'
and Fournier's book.~\cite[Theorem~7.47]{adams2003}) For this reason local
estimates can be summed up to obtain global ones.

For \emph{non-integers} $s>0$ we \emph{define} the Sobolev spaces of fractional
order as $W^{s,p}(\gdom) := B_p^{s,p}(\gdom)$. For $p\neq\infty$ these spaces
coincide with the Sobolev--Slobodeckij spaces,~\cite[Theorem~6.7]{devore1993}
but unless $p=2$ they differ from the fractional order Sobolev spaces obtained
by the \enquote{complex} interpolation method as defined by Adams and Fournier.
For integer values $s=k$ the Besov spaces $B_p^{k,p}(\gdom)$ do not coincide
with  $W^{k,p}(\gdom)$, except for the pathological case $p=2$.~\cite[Section~7.33]{adams2003}
However, one always has $B_1^{s,p}(\gdom)\hookrightarrow W^{s,p}(\gdom)\hookrightarrow
B_{\infty}^{s,p}(\gdom)$. For this reason, we will often first establish our
results for all integer values $s=k$ and then conclude by interpolation to the
intermediate spaces. 

The intermediate spaces with negative index are defined via interpolation,
analogously to the positive case: $W^{-\theta n,p}(\gdom):=\bigl[\bigl(L^q(\gdom)\bigr)',
\bigl(W^{n,q}(\gdom)\bigr)'\bigr]_{\theta,p}$. For $p\neq 1$, i.\,e.,
$q\neq\infty$, it can be shown that they in fact are the dual spaces of the
corresponding intermediate spaces with positive index: $W^{-s,p}(\gdom) =
\bigl(W^{s,q}(\gdom)\bigr)'$.~\cite[Theorem~3.7.1]{bergh1976} For this reason,
we will sometimes exclude the case $p=1$. In summary, the spaces $W^{s,p}(\gdom)$
are defined for all $s\in\mathbb{R}$, $1\leq p\leq\infty$.

\subsection{Sobolev Embeddings and Stein Extension}
Before moving on to the spline spaces, we recall the Stein extension theorem~%
\cite[Chapter VI, Theorem 5]{stein1970}: there exists a linear extension
operator $E$ that fulfills $\Vert Eu\Vert_{W^{s,p}(\wholespace)} \lesssim
\Vert u\Vert_{W^{s,p}(\gdom)}$ for all $u\in W^{s,p}(\gdom)$, $s\geq 0$,
$1\leq p\leq\infty$. We also will use the following variant of the Sobolev
embedding theorem: let $s > \tfrac{\spdim}{p}$ or $s = \spdim$ if $p=1$.
Then $W^{s,p}(\gdom)\hookrightarrow C(\gdom)$ and $\Vert u\Vert_{L^\infty(\gdom)}
\lesssim\Vert u\Vert_{W^{s,p}(\gdom)}$. Taking into account the secondary index
of Besov spaces, for $p\neq\infty$ one can refine the embedding to
$B_1^{\frac{\spdim}{p},p}(\gdom)\hookrightarrow C(\gdom)$.~\cite[Theorem~4.12
and Theorem~7.34]{adams2003}

\subsection{Spline Spaces}
The spline spaces will be defined on uniform Cartesian grids, which we
introduce first, after which we define the spline spaces and recall some
of their properties from approximation theory.
\begin{definition}[Cartesian Grid and Fictitious Domains]
Let $\sigma>0$ be given. With each $\mathbf{i}\in\mathbb{Z}^\spdim$
we associate a Cartesian grid-point $\vv{x_i}:= (i_1\sigma, i_2\sigma,\ldots,i_\spdim\sigma)^\top$
and an element $Q_\vv{i}:=\prod_{d=1}^{\spdim}\bigl(i_d\sigma,\,(i_d+1)\sigma\bigr)$.
We define the fictitious domain $\Omega_\sigma$ as the union of all elements that
intersect the physical domain $\Omega$. Furthermore we define cut and uncut
elements $\Omega_\sigma^\Gamma$ and  $\Omega_\sigma^\circ$, respectively:
\begin{equation}
\begin{split}
\Omega_\sigma        &:= \inte\ \bigcup\bigl\lbrace \clos{Q_\vv{i}}\,\bigr|\,
\meas_\spdim(Q_\vv{i}\cap\Omega)>0\bigr\rbrace,\\
\Omega_\sigma^\Gamma &:= \inte\ \bigcup\bigl\lbrace \clos{Q_\vv{i}}\,\bigr|\,
Q_\vv{i}\in\Omega_\sigma\wedge Q_\vv{i}\not\subset\Omega\bigr\rbrace,\\
\Omega_\sigma^\circ &:= \inte\ \bigcup\bigl\lbrace \clos{Q_\vv{i}}\,\bigr|\,
Q_\vv{i}\in\Omega_\sigma\wedge Q_\vv{i}\subset\Omega\bigr\rbrace.
\end{split}
\end{equation}
The stabilization will make use of the following set of faces near the boundary:
\begin{equation}
\mathcal{F}_\sigma := \lbrace F\text{ is a face of some element } Q_\vv{i}\in\Omega^\Gamma_\sigma
\text{ and } F\notin\partial\Omega_\sigma\rbrace.
\end{equation}
\end{definition}
Here and in what follows we write $Q_\vv{i}\in\Omega_\sigma,
Q_\vv{i}\in\Omega_\sigma^\Gamma$, and $Q_\vv{i}\in\Omega_\sigma^\circ$ to refer
to the elements these domains are composed of. An illustration of these definitions
is given in \cref{fig:fictitious_domains}. We will make use of the following
somewhat technical assumption: for every $Q_\vv{i}\in\Omega_\sigma^\Gamma$ there
exists a finite sequence $(F_{\vv{i},1},F_{\vv{i},2},\ldots,F_{\vv{i},K})\subset
\mathcal{F}_\sigma$ such that the following conditions are fulfilled: every two
subsequent $F_{\vv{i},k}$ and $F_{\vv{i},k+1}$ are faces of a single element
$Q_\vv{j}$, the number $K$ is bounded independent of $\sigma$ and $\vv{i}$, and
the last face $F_{\vv{i},K}$ belongs to an uncut element $Q_\vv{i}\in\Omega_\sigma^\circ$.
This assumption means that uncut cells can always be reached from cells in
$\Omega_\sigma^\Gamma$ by crossing a bounded number of faces. For sufficiently
small $\sigma$ this condition is often fulfilled with $K=1$; if necessary it can
be enforced by moving additional elements from $\Omega_\sigma^\circ$ to
$\Omega_\sigma^\Gamma$. 

\begin{figure}\centering%
\includegraphics{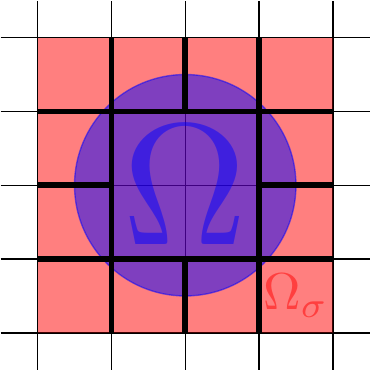}%
\caption{\label{fig:fictitious_domains}An illustration of the fictitious domain
approach. The domain $\Omega$ ({\color{blue}blue}), in this case a circle, may
intersect the infinite Cartesian grid in an arbitrary manner. The fictitious
domain~$\Omega_\sigma$~({\color{red}red}) is defined as the union of all
intersected cells. The domain $\Omega_\sigma^\circ$ in this case consists of the
four elements entirely lying in $\Omega$. The set of faces $\mathcal{F}_\sigma$
is highlighted using bold lines. It can be thought of forming a bridge between
$\Omega_\sigma^\circ$ and the remaining elements in $\Omega_\sigma^\Gamma$.}%
\end{figure}%

\begin{definition}[Spline Spaces]
Given $n\in\mathbb{N}$ and $1\leq p\leq\infty$, we define the tensor product
spline space on the Cartesian grid, equipped with the $L^p$-norm:
\begin{equation}
\begin{split}
V_\sigma^{n,p}(\wholespace) & := \bigl\lbrace
f:\wholespace\to\reals\,\big|\,f|_{Q_\vv{i}}\in\mathbb{Q}_{n-1}, \vv{i}\in\mathbb{Z}^\spdim\bigr\rbrace\cap
C_0^{n-2}(\wholespace), \\
\Vert\cdot\Vert_{V_\sigma^{n,p}(\wholespace)} &:=\Vert\cdot\Vert_{L^p(\wholespace)},
\end{split}
\end{equation}
where the symbol $\mathbb{Q}_{n-1}$ refers to the space of polynomials of
\textbf{coordinate-wise degree} $n-1$ or less. For $\gdom\subset\wholespace$,
we define $V_\sigma^{n,p}(\gdom)$ by restriction from $\wholespace$ to $\gdom$.
In analogy to the Sobolev spaces, the normed dual of $V_\sigma^{n,q}(\gdom)$
will be denoted by $V_\sigma^{-n,p}(\gdom)$, where $q$ denotes the Hölder
conjugate to $p$. Denoting as usual the duality paring by $\langle\cdot,\cdot\rangle$,
its norm is thus given by:
\begin{equation}
\Vert\cdot\Vert_{V_\sigma^{-n,p}(\gdom)} :=
\sup_{v_\sigma\in V_\sigma^{n,q}(\gdom)} \frac{\langle\cdot, v_\sigma\rangle}{\Vert v_\sigma\Vert_{V_\sigma^{n,q}(\gdom)}} =
\sup_{v_\sigma\in V_\sigma^{n,q}(\gdom)} \frac{\langle\cdot, v_\sigma\rangle}{\Vert v_\sigma\Vert_{L^q(\gdom)}}.
\end{equation}
\end{definition}

For a fixed bounded domain $\gdom$ and fixed values of $\sigma>0$ and
$n\in\mathbb{N}$, these spaces of course all have the same topology and are in
this sense independent of $p$. In the next sections this notation will prove to
be useful, in the other cases the index $p$ will be omitted. It is well-known
that for $n\in\mathbb{N}$, $1\leq p\leq\infty$, one has $V_\sigma^{n}(\gdom)
\subset W^{n-1,p}(\gdom)$. For $p\neq\infty$ this can be improved to $V_\sigma^n
(\gdom)\subset B_{\infty}^{\hat{s},p}(\gdom)$, $\hat{s}:= n - 1 + \tfrac{1}{p} =
n - \tfrac{1}{q}$, and furthermore $V_\sigma^n(\gdom) \subset W^{s,p}(\gdom)$ for
all $0\leq s <\hat{s}$.~\cite{devore1988,devore1993} In particular, the spaces
$V_\sigma^n(\gdom)$ are not, but \enquote{almost} are embedded in $W^{n,1}(\gdom)$.

\subsection{Some Properties of Splines}We now introduce some important basic
properties of the spline spaces, of which our analysis will make frequent use.
For proofs of these result, we refer to Schumaker’s book.~\cite{schumaker2007}
The B-splines form a particularly useful basis for the spaces $V_\sigma^n(\gdom)$.
\begin{definition}[B-Splines]
The cardinal B-splines $b^n:\reals\to\reals$, $n\in\mathbb{N}$, are
defined recursively via:
\begin{equation}
b^1(x) := \begin{cases} 1 & x\in[0,1), \\ 0 & \text{else}, \end{cases}\quad 
b^n(x) := \bigl(b^{n-1}\star b^1\bigr)(x) = \idx{0}{1}{b^{n-1}(x-y)}{y}.
\end{equation}
Reusing the symbol $b^n$, the corresponding multivariate B-splines are
defined coordinate-wise as $b^n(\vv{x}):=\prod_{d=1}^{\spdim}b^n(x_d)$. For a
given $\sigma>0$, with each Cartesian grid point $\vv{x_i}$, $\vv{i}\in
\mathbb{Z}^\spdim$, we associate the shifted and scaled B-spline
$b_{\sigma,\vv{i}}^n(\vv{x}):= b^n(\tfrac{\vv{x}-\vv{x}_i}{\sigma})$. For a
given domain $\gdom\subset\wholespace$ the corresponding index set is defined
as:
\begin{equation}\label{eqn:indexxset}
\Lambda_\sigma^n(\gdom) := \bigl\lbrace\vv{i}\in\mathbb{Z}^\spdim\ \big| 
\meas_{\spdim}\bigl(\supp{b_{\sigma,\vv{i}}^n}\cap\gdom\bigr)>0\bigr\rbrace.
\end{equation}
\end{definition}

This basis has many desirable properties, among which are the smallest possible
support of its members $\supp b_{\sigma,\vv{i}}^{n} = \prod_{d=1}^{\spdim}[i_d\sigma,
(i_d + n)\sigma]$, their positivity $0\leq b_{\sigma,\gv{\lambda}}^n \leq 1$,
the fact that they form a partition of unity $\sum_{\vv{i}\in\mathbb{Z}^\spdim}
b_{\sigma,\vv{i}}^n \equiv 1$, and most importantly the norm equivalence that
follows.

In what follows, the symbol $\qdom$ will refer to an arbitrary finite collection
of entire cubes from the Cartesian grid, e.\,g., the domains $\Omega_\sigma$,
$\Omega_\sigma^\circ$, or $\Omega_\sigma^\Gamma$. On such domains one can show
some very useful properties, of which we will make frequent use. For proofs of
these results and further references we refer the reader to Schumaker's book.
\cite{schumaker2007}

\begin{lemma}[Stability of the B-Spline Basis]
Let $\qdom\subset\wholespace$ be a finite collection of entire, uncut cubes
$Q_\vv{i}$ from the Cartesian grid of size $\sigma>0$. Then every function
$v_\sigma\in V_\sigma^{n}(\qdom)$, $n\in\mathbb{N}$, can be written as
\begin{equation}
u_\sigma = \sum_{\gv{\lambda}\in\Lambda_{\sigma}^n(\qdom)}
\mathsf{u}_{\sigma,\gv{\lambda}}b_{\sigma,\gv{\lambda}}^n,
\end{equation}
with a uniquely determined coefficient vector $\mathsf{u}_{\sigma} =
\bigl(\mathsf{u}_{\sigma,\gv{\lambda}}\bigr)_{%
\gv{\lambda}\in\Lambda_{\sigma}^n(\qdom)}\in \ell^p(\Lambda_{\sigma}^n(\qdom)) =
\reals^{\dim V_\sigma^{n,p}(\qdom)}$. The $L^p$- and $\ell^p$-norms of respectively
$u_\sigma$ and $\sv{u}_\sigma$ are equivalent for $1\leq p\leq \infty$:
\begin{equation}\label{eqn:Bstab}
\sigma^{\frac{\spdim}{p}}\Vert\sv{u}_\sigma\Vert_{\ell^p} \lesssim \Vert u_\sigma\Vert_{L^p(\qdom)} \lesssim
\sigma^{\frac{\spdim}{p}}\Vert\sv{u}_\sigma\Vert_{\ell^p}.
\end{equation}
\end{lemma}
\begin{lemma}[Inverse Estimates]\label{thm:inverseineq}
Let $u_\sigma\in V_\sigma^{n}(\qdom)$, $n\in\mathbb{N}$. On every element $Q\in\qdom$
and for all $1\leq p_1, p_2\leq\infty$, $0\leq r\leq s$, the following local inequality
holds:
\begin{equation}\label{eqn:localinverseineq}
\Vert u_\sigma\Vert_{W^{s,p_1}(Q)} \lesssim
\sigma^{\frac{\spdim}{p_1}-\frac{\spdim}{p_2}}
\sigma^{r-s}
\Vert u_\sigma\Vert_{W^{r,p_2}(Q)}.
\end{equation}
Globally one has for all $1\leq p_1,p_2\leq\infty$, $0\leq s < n - 1 +
\min\lbrace\tfrac{1}{p_1},\tfrac{1}{p_2}\rbrace$ or $s=n-1$:
\begin{equation}\label{eqn:globalinverseineqpq}
\Vert u_\sigma\Vert_{W^{s,p_1}(\qdom)} \lesssim
\sigma^{\min\lbrace 0,\frac{\spdim}{p_1}-\frac{\spdim}{p_2}\rbrace}
\Vert u_\sigma\Vert_{W^{s,p_2}(\qdom)},
\end{equation}
and for all $1\leq p\leq\infty$, $0\leq r\leq s < n - 1 + \tfrac{1}{p}$ or
$s = n -1$:
\begin{equation}\label{eqn:globalinverseineqrs}
\Vert u_\sigma\Vert_{W^{s,p}(\qdom)} \lesssim
\sigma^{r-s}
\Vert u_\sigma\Vert_{W^{r,p}(\qdom)}.
\end{equation}
\end{lemma}
\begin{lemma}[Quasi-interpolator]\label{thm:quasiinterpolant}
For every $n\in\mathbb{N}$ there exists a projection operator $P_\sigma^n:
L^1(\qdom)\to V_\sigma^n(\qdom)$ called the \emph{quasi-interpolator}. For
all $u\in W^{k,p}(\qdom)$, $k\in\mathbb{N}_0$, $0\leq k\leq n$ this operator
fulfills:
\begin{align}\label{eqn:Pconsistency}
\Vert u - P_\sigma^n u\Vert_{W^{l,p}(Q)}
& \lesssim 
\sigma^{k-l}\vert u\vert_{W^{k,p}(\hat{Q})} & l\in\lbrace 0,\ldots,k\rbrace, Q\in\qdom,\\
\vert P_\sigma^n u\vert_{W^{l,p}(Q)}
& \lesssim
\sigma^{k-l}\vert u\vert_{W^{k,p}(\hat{Q})} & l\in\lbrace k,\ldots,n\rbrace, Q\in\qdom,
\label{eqn:Pcontinuity}
\end{align}
where $\hat{Q}=\qdom\cap\bigcup_{\gv{\lambda}\in\Lambda_\sigma^n(Q)}
\supp b_{\gv{\lambda},\sigma}^n$ is the union of the supports of all the
B-splines that do not vanish on $Q$. Moreover, for arbitrary
$0\leq r < s\leq n$, $1\leq p_1^\prime, p_2^\prime\leq\infty$ it holds
that:
\begin{equation}
\Vert u - P_\sigma^n u\Vert_{B^{r,p}_{p_1^\prime}(Q)}
\lesssim 
\sigma^{s-r}\Vert u\Vert_{B^{s,p}_{p_2^\prime}(\hat{Q})}.
\end{equation}
\end{lemma}
\begin{remark}
This result can be improved in the sense that the right hand side of the above
inequality \eqref{eqn:Pconsistency} only needs to involve \enquote{pure}
derivatives in the coordinate directions.~\cite{dahmen1980} We will not be able
to use this fact, however, because our analysis will rely on the Stein extension
theorem, which is formulated for the usual Sobolev spaces involving mixed
derivatives.
\end{remark}

On domains $\qdom$ the $L^2(\qdom)$-projection is bounded as an operator from
$L^p(\qdom)\to V_\sigma^{n,p}(\qdom)$, $1\leq p\leq\infty$.~\cite{douglas1974,
crouzeix1987} From this fact and the stability of the B-spline basis it is an
easy task to derive the following lemma.
\begin{lemma}\label{thm:equivalence}
Every functional $f\in V_\sigma^{-n,p}(\qdom)$, $1\leq p\leq\infty$, has a unique
representative $f_\sigma\in V_\sigma^{n,p}(\qdom)$ such that:
\begin{equation}
\langle f,v_\sigma\rangle = \idx{\qdom}{}{f_\sigma v_\sigma}{x} \qquad
\forall v_\sigma\in V_\sigma^{n,q}(\qdom).
\end{equation}
The norms of $f$ and $f_\sigma$ are equivalent:
\begin{equation}
\Vert f_\sigma\Vert_{L^p(\qdom)}
\lesssim
\Vert f\Vert_{V_\sigma^{-n,p}(\qdom)}
\leq
\Vert f_\sigma\Vert_{L^p(\qdom)}.
\end{equation}
\end{lemma}

\section{Particle Initialization}\label{sec:initialization}
In this section we discuss how to construct particle approximations of spline
functions $\tilde{u}_h\in V_h^{n,p}(\Omega)$. A particle approximation of
general functions $u\in L^p(\Omega)$ can then be obtained by setting $\tilde{u}_h$
to a suitable approximation of $u$.

\subsection{Particle Approximations of Splines}\label{subsec:construction}
As before, let $\Omega\subset\wholespace$ denote an open, bounded Lipschitz
domain and let $\tilde{u}_h\in V_h^n(\Omega)$, $h>0$, $n>\spdim$, denote the
spline we want to approximate by a particle field. The condition $n>\spdim$
ensures that $V_h^n(\Omega)\subset W^{\spdim,p}(\Omega)$ for $1\leq p\leq\infty$,
which will simplify the analysis. It is likely that similar results can be
obtained for smaller choices of $n$ at the cost of a more technical analysis,
but we see no clear benefit from this. For each $\gv{\lambda}\in\Lambda_h^n(\Omega)$
we chose quadrature nodes $\vv{x}_{\gv{\lambda},i}\in\Omega\cap\supp b^n_{%
h,\gv{\lambda}}$ and associated weights $w_{i,\gv{\lambda}}\in\reals$, $i=1,\ldots,N_{%
\gv{\lambda}}$, such that:
\begin{eqnarray}
\sum_{i=1}^{N_{\gv{\lambda}}}|w_{i,\gv{\lambda}}|&\leq& C_{\mathrm{Stab}}(nh)^\spdim, \\
\sum_{i=1}^{N_{\gv{\lambda}}}w_{i,\gv{\lambda}}%
b^n_{h,\gv{\lambda}}(\vv{x}_{\gv{\lambda},i})
b^n_{h,\gv{\mu}}(\vv{x}_{\gv{\lambda},i}) &=&
\idx{\Omega}{}{b^n_{h,\gv{\lambda}}b^n_{h,\gv{\mu}}}{\vv{x}}\qquad
\forall\gv{\mu}\in\Lambda_h^n(\Omega),
\end{eqnarray}
where $C_{\mathrm{Stab}}\geq 1$ is a user-defined stability constraint.

All that is required to construct such quadrature rules is a mesh of the boundary.
For B-splines $b^n_{h,\gv{\lambda}}$ whose support entirely lies in $\Omega$ one
can choose standard Gau{\ss}--Legendre quadrature rules on each cell $Q\in\supp
b^n_{h,\gv{\lambda}}\subset\Omega$. For B-splines with cut support the main 
difficulty is to compute the integrals on the right for each $\gv{\mu}$ with
$\supp b^n_{h,\gv{\lambda}}\cap \supp b^n_{h,\gv{\mu}} \cap\Omega\neq\emptyset$.
This can be done using the boundary mesh: note that the product
$b^n_{h,\gv{\lambda}}b^n_{h,\gv{\mu}}$ is again of the form $\prod_{d=1}^{\spdim}
P_d(x_d)$, where the $P_d$ are certain one-dimensional, piecewise polynomials with
global smoothness $W^{n-1,\infty}$. The ordinary, one-dimensional anti-derivative
$\mathcal{P}_1(x):=\idx{-\infty}{x}{P_1(x')}{x'}$ of -- for example -- $P_1$ is
known explicitly. We define the vector-valued function
\begin{math}
\vv{F}(\vv{x}) :=
\left(
\mathcal{P}_1(x_1)\prod_{d=2}^{\spdim}P_d(x_d),\ 0,\ 0,\ \ldots,\ 0
\right)^\top,
\end{math}
and note that $\nabla\cdot\vv{F}=\prod_{d=1}^{\spdim}P_d(x_d) =
b^n_{h,\gv{\lambda}}b^n_{h,\gv{\mu}}$. The divergence theorem thus allows us to
convert the volume integral to a boundary integral:
\begin{equation}
\idx{\Omega}{}{b^n_{h,\gv{\lambda}}b^n_{h,\gv{\mu}}}{\vv{x}} =
\idx{\Omega}{}{\nabla\cdot\vv{F}}{\vv{x}} =
\idx{\pd\Omega}{}{\mathcal{P}_1(x_1)(\vv{e}_1\cdot\vv{n})\prod_{d=2}^{\spdim}
P_d(x_d)}{S(\vv{x})}.
\end{equation}
On each patch of the boundary mesh the integrand on the right is $W^{n-1,\infty}$-smooth.
The integral can thus be efficiently approximated with standard quadrature rules on the
boundary mesh. This is similar to the approach of Duczek and Gabbert,~\cite{duczek2015}
who successfully applied it to less smooth shape functions. Once the integrals
have been computed, quadrature rules can for example be constructed using the
following procedure:
\begin{enumerate}
	\item Randomly scatter (additional) points $\vv{x}_{\gv{\lambda},i}$ over
	      $\Omega\cap\supp b^n_{h,\gv{\lambda}}$.
    \item Solve the following linear programming problem for the unknown weights
          $w_{\gv{\lambda},i}$:
    \begin{equation}
    \begin{split}
	\min &\sum_{i=1}^{N_{\gv{\lambda}}}w_{\gv{\lambda},i}\\
	\sum_{i=1}^{N_{\gv{\lambda}}}w_{i,\gv{\lambda}}%
	b^n_{h,\gv{\lambda}}(\vv{x}_{\gv{\lambda},i})
	b^n_{h,\gv{\mu}}(\vv{x}_{\gv{\lambda},i}) &=
	\idx{\Omega}{}{b^n_{h,\gv{\lambda}}b^n_{h,\gv{\mu}}}{\vv{x}}\qquad
	\forall\gv{\mu}\in\Lambda_h^m(\Omega),\\
	w_{\gv{\lambda},i}\geq 0
    \end{split}
    \end{equation}
    \item If no solution exists that fulfils the stability criterion, go to step~1
    and repeat.
\end{enumerate}
Given such quadrature nodes and weights, let us denote by $(c_{\gv{\lambda}})_{%
\gv{\lambda}\in\Lambda_h^n(\Omega)}$ the B-spline coefficients of $\tilde{u}_h$
such that $\tilde{u}_h=\sum_{\gv{\lambda}\in\Lambda_h^n(\Omega)}c_{\gv{\lambda}}
b^n_{h,\gv{\lambda}}$. We then define the particle approximation $u_h$ as follows:
\begin{equation}\label{eqn:particleapprox}
u_h := \sum_{\gv{\lambda}\in\Lambda_h^n(\Omega)}\sum_{i=1}^{N_{\gv{\lambda}}}
w_{\gv{\lambda},i}c_{\gv{\lambda}}b^n_{h,\gv{\lambda}}(\vv{x}_{\gv{\lambda},i})\delta_{\vv{x}_{\gv{\lambda},i}}.
\end{equation}

\subsection{Error Bounds}
We will identify the function $\tilde{u}_h\in V_h^{n,p}(\Omega)$ with the functional
\begin{equation}
\langle \tilde{u}_h, \varphi \rangle := \idx{\Omega}{}{\tilde{u}_h\varphi}{\vv{x}}\qquad
\forall\varphi\in L^q(\Omega).
\end{equation}
The key result of this section is the following theorem.
\begin{theorem}\label{thm:particleerror}
The particle approximation $u_h$ from \eqref{eqn:particleapprox} fulfills for
all $1 < p\leq\infty$, $s>\tfrac{\spdim}{q}$ or $s=\spdim$ if $q=1$:
\begin{equation}\label{eqn:particlebound}
\Vert u_h\Vert_{W^{-s,p}(\Omega)}\lesssim h^{\frac{\spdim}{q}}\Vert\tilde{u}_h\Vert_{L^p(\Omega_h)}.
\end{equation}
Moreover, for all $\tfrac{\spdim}{q}<s\leq n$ or $s=\spdim$ if $q=1$
the following error bound holds:
\begin{equation}\label{eqn:particleerrbound}
\Vert\tilde{u}_h-u_h\Vert_{W^{-s,p}(\Omega)}\lesssim h^s\Vert\tilde{u}_h\Vert_{L^p(\Omega_h)}.
\end{equation}
\end{theorem}
\begin{proof}
For arbitrary $\varphi\in W^{s,q}(\Omega)\hookrightarrow C(\Omega)$ it holds
by H\"older's inequality that:
\begin{multline}
\bigl|\langle u_h,\varphi\rangle\bigr| =
\biggl|\sum_{\gv{\lambda}\in\Lambda_h^n(\Omega)}\sum_{i=1}^{N_{\gv{\lambda}}}
w_{\gv{\lambda},i}c_{\gv{\lambda}}b^n_{h,\gv{\lambda}}(\vv{x}_{\gv{\lambda},i})\varphi(\vv{x}_{\gv{\lambda},i})\biggr| \\
\leq \sum_{\gv{\lambda}\in\Lambda_h^m(\Omega)}
\Vert\tilde{u}_h\Vert_{L^\infty(\supp b^n_{h,\gv{\lambda}})}
\Vert\varphi\Vert_{L^{\infty}(\supp b^n_{h,\gv{\lambda}}\cap\Omega)}
\underbrace{\sum_{i=1}^{N_{\gv{\lambda}}}|w_{\gv{\lambda},i}|}_{\lesssim h^{\spdim}}.
\end{multline}
Again applying H\"older's inequality, an inverse inequality for $\tilde{u}_h$,
and the Sobolev embedding on $\varphi$ this yields with the usual modifications
for $p=\infty$:
\begin{multline}
\bigl|\langle u_h,\varphi\rangle\bigr|\lesssim h^{\frac{\spdim}{q}}
\biggl(\sum_{\gv{\lambda}\in\Lambda_h^n(\Omega)}\Vert\tilde{u}_h\Vert_{L^p(\supp b^n_{h,\gv{\lambda}})}^p\biggr)^{\frac{1}{p}}
\biggl(\sum_{\gv{\lambda}\in\Lambda_h^n(\Omega)}\Vert\varphi\Vert_{W^{s,q}(\supp b^n_{h,\gv{\lambda}}\cap\Omega)}^q\biggr)^{\frac{1}{q}} \\
\lesssim h^{\frac{\spdim}{q}}\Vert\tilde{u}_h\Vert_{L^p(\Omega_h)}\Vert\varphi\Vert_{W^{s,q}(\Omega)}.
\end{multline}
This proves \eqref{eqn:particlebound}. For the error bound first note that:
\begin{equation}
\langle\tilde{u}_h-u_h,\varphi\rangle =
\underbrace{\langle\tilde{u}_h-u_h,P_{h}^{n}E\varphi\rangle}_{=0} +
\langle\tilde{u}_h,\varphi-P_{h}^{n}E\varphi\rangle - \langle u_h,\varphi-P_{h}^{n}E\varphi\rangle,
\end{equation}
where for the second term we have:
\begin{equation}
\langle\tilde{u}_h,
\varphi-P_h^nE\varphi\rangle\leq \Vert\tilde{u}_h\Vert_{L^p(\Omega_h)}
\Vert E\varphi-P_h^nE\varphi\Vert_{L^q(\Omega_h)}
\lesssim h^s\Vert\tilde{u}_h\Vert_{L^p(\Omega_h)}\Vert\varphi\Vert_{W^{s,q}(\Omega)}.
\end{equation}
For the last term we obtain analogous to the proof of \eqref{eqn:particlebound} that
\begin{equation}
\langle u_h,\varphi - P_{h}^{n}E\varphi\rangle \lesssim
h^{\frac{\spdim}{q}}\Vert\tilde{u}_h\Vert_{L^p(\Omega_h)}
\biggl(\sum_{\gv{\lambda}\in\Lambda_h^n(\Omega)}\Vert E\varphi-
P_h^nE\varphi\Vert_{L^{\infty}(\supp b^n_{h,\gv{\lambda}})}^q\biggr)^{\frac{1}{q}}.
\end{equation}
We now apply the Sobolev embedding for the Besov space $B_1^{\frac{\spdim}{q},q}
(\Omega_h)\hookrightarrow C(\Omega_h)$ and obtain using \cref{thm:quasiinterpolant}:
\begin{equation}
\Vert E\varphi-P_h^nE\varphi\Vert_{L^{\infty}(\supp b^n_{h,\gv{\lambda}})}
\lesssim
\Vert E\varphi-P_h^nE\varphi\Vert_{B^{\frac{\spdim}{q},q}_1(\supp b^n_{h,\gv{\lambda}})}
\lesssim h^{s-\frac{\spdim}{q}}\Vert E\varphi\Vert_{W^{s,q}(\widehat{\supp b^n_{h,\gv{\lambda}})}},
\end{equation}
where $\widehat{\supp b^n_{h,\gv{\lambda}}} = \bigcup_{\gv{\mu}\in\Lambda_h^n(\supp b^n_{h,\gv{\lambda}})}
\supp b^n_{h,\gv{\mu}}$. Inequality \eqref{eqn:particleerrbound} now again follows
by a finite overlap argument.
\end{proof}

One of the key features of this result is the fact that these estimates only depend 
on $L^p$-norms of the spline $\tilde{u}_h$, similar to the results of Cohen and
Perthame.~\cite{cohen2000} Previous estimates have mostly been of the form:
$\Vert u - u_h\Vert_{W^{-s,p}(\Omega)}\lesssim h^s\Vert u\Vert_{W^{s,p}(\Omega)}$,
\cite[Theorem A.1.1]{cottet2000} suggesting that there might be room for
improvement to $h^{2s}$. This is not the case, and we believe that this fact is
not well-known in the particle method communities. We therefore recall a theorem
of Bakhvalov,~\cite[Chapter~4, §3]{sobolev1997} which indicates that these
estimates are in fact optimal in terms of convergence order.
\begin{theorem}[Bakhvalov]
Let $\Omega\subset\wholespace$ be a bounded Lipschitz domain, $n\in\mathbb{N}$,
$n>\tfrac{\spdim}{2}$, and let $u\equiv 1$. Let $u_{h,k} = \sum_{i=1}^{N_k} U_{i,k}
\delta_{\vv{x}_{i,k}}$, $k=1,2,\ldots$, denote a sequence of particle
approximations of $u$ such that $N_k\to\infty$ as $k\to\infty$ and let us
define the average particle spacing as $h:=h(k):= \sqrt[\spdim]{\meas_{\spdim}(\Omega)/N_k}$.
Then \textbf{for every such sequence} one has:
\begin{equation}
\Vert u - u_{h,k}\Vert_{W^{-n,2}(\Omega)} \gtrsim h^n \qquad k\to\infty,
\end{equation}
with the hidden constant independent of $k$.
\end{theorem}

Noting that the error bounds only depend on the $L^p(\Omega_h)$-norm of the
function $\tilde{u}_h$, this constraint can to some extent be bypassed by
choosing $n$ very large. This would later allow one to chose the smoothing
length $\sigma$ essentially proportional to $h$. On the other hand, the
hidden constants in the $\lesssim$-notation get larger as $n$ grows.
Furthermore, this approach would require the use of equally smooth trial
spaces. The spline spaces that we are going to employ for regularization in
the next section only have finite smoothness, however.

\section{Particle Regularization}\label{sec:regularization}
Let $n\in\mathbb{N}$ and $1\leq p\leq\infty$. Our approach will make use of the
following operators:
\begin{align}
A: V_\sigma^{n,p}(\Omega_\sigma)\to V_\sigma^{-n,p}(\Omega_\sigma),\quad
\langle Au_\sigma, v_\sigma\rangle &:= \idx{\Omega}{}{u_\sigma v_\sigma}{\vv{x}}, \\
J: V_\sigma^{n,p}(\Omega_\sigma)\to V_\sigma^{-n,p}(\Omega_\sigma),\quad
\langle Ju_\sigma, v_\sigma\rangle &:= \sigma^{2n-1}\sum_{F\in\mathcal{F}_\sigma}
\idx{F}{}{\left\llbracket\pdx{^{n-1}u_\sigma}{\vv{n}_F^{n-1}}\right\rrbracket
\left\llbracket\pdx{^{n-1}v_\sigma}{\vv{n}_F^{n-1}}\right\rrbracket}{S},
\end{align}
and $A_\varepsilon := A + \varepsilon J$, where $\varepsilon>0$ denotes a
user-defined stabilization parameter. The symbol $\llbracket\cdot\rrbracket$
refers to the jump operator; it is the difference of the one-sided traces
on a face $F$. $\vv{n}_F$ stands for the face's normal vector, which in
our case always coincides with some Cartesian basis vector: $\vv{n}_F\in
\lbrace{\vv{e}_1,\vv{e}_2,\ldots,\vv{e}_{\spdim}}\rbrace$. The stabilization
operator $J$ will be called the \emph{ghost penalty}.~\cite{burman2010} The
operator $A$ effectively restricts a function from $\Omega_\sigma$ to $\Omega$.
We will establish that its stabilized version $A_\varepsilon$ is invertible,
yielding the \emph{approximate extension operator $A_\varepsilon^{-1}$}.

\subsection{Continuity and Consistency}
For $\varepsilon=0$ the approximate extension operator $A_0^{-1}$ is the 
$L^2(\Omega)$-projection onto $V_\sigma^{n}(\Omega_\sigma)$. For $\varepsilon>0$,
however, $A_\varepsilon^{-1}$ ceases to be a projection, but the
difference to $A_0^{-1}$ is small; a fact we will call \emph{consistency}.
The main difference to previous analyses of the ghost penalty operator is
that we consider results in $L^p$-spaces for $p\neq 2$.  
\begin{lemma}\label{thm:Jcontinuity}
The ghost-penalty operator $J$ is continuous. In other words, for all
$u_\sigma\in V_\sigma^{n,p}(\Omega_\sigma)$, $n\in\mathbb{N}$, $1\leq p\leq\infty$,
it holds that:
\begin{equation}\label{eqn:Jcontinuity}
\Vert Ju_\sigma\Vert_{V^{-n,p}(\Omega_\sigma)}\lesssim \Vert u_\sigma\Vert_{L^p(\Omega_\sigma)}.
\end{equation}
Moreover, for any $u\in W^{s,p}(\Omega_\sigma)$, $0\leq s\leq n$, the
quasi-interpolant $P_\sigma^n u$ of $u$ fulfils:
\begin{equation}\label{eqn:Jconsistency}
\Vert JP_\sigma^n u\Vert_{V^{-n,p}(\Omega_\sigma)} \lesssim
\sigma^{s}\Vert u\Vert_{W^{s,p}(\Omega_\sigma)}.
\end{equation}
\end{lemma}
\begin{proof}
For arbitrary $u_\sigma\in V_\sigma^{n,p}(\Omega_\sigma)$ and
$v_\sigma\in V_\sigma^{n,q}(\Omega_\sigma)$ we obtain by repeatedly using
Hölder's and the triangular inequality:
\begin{multline}
\bigl|\bigl\langle Ju_\sigma, v_\sigma \bigr\rangle\bigr| = \sigma^{2n-1}
\left|\sum_{F\in\mathcal{F}_\sigma}\idx{F}{}{\left\llbracket\pdx{^{n-1}u_\sigma}{\vv{n}_F^{n-1}}\right\rrbracket
\left\llbracket\pdx{^{n-1}v_\sigma}{\vv{n}_F^{n-1}}\right\rrbracket}{S}\right| \\
\leq \sigma^{2n-1}
\left(\sum_{F\in\mathcal{F}_\sigma}\left\Vert\left\llbracket
\pdx{^{n-1}u_\sigma}{\vv{n}_F^{n-1}} \right\rrbracket\right\Vert_{L^p(F)}^p\right)^{\frac{1}{p}}
\left(\sum_{F\in\mathcal{F}_\sigma}\left\Vert\left\llbracket
\pdx{^{n-1}v_\sigma}{\vv{n}_F^{n-1}} \right\rrbracket\right\Vert_{L^q(F)}^q\right)^{\frac{1}{q}}
\end{multline}
with the usual modifications for $p=\infty$ or $q=\infty$. For arbitrary $w\in
W^{1,p}(Q)$, $1\leq p\leq\infty$, $Q\in\Omega_\sigma$ an arbitrary cube from the
Cartesian grid, we have the trace estimate $\Vert w\Vert_{L^p(\pd Q)}\lesssim
\Vert w\Vert_{L^p(Q)}^{\frac{1}{q}}\Vert w\Vert_{W^{1,p}(Q)}^{\frac{1}{p}}$.~%
\cite[Lemma~(1.6.6)]{brenner2008} Together with an inverse estimate this leads
to:
\begin{align}
\left\Vert\left\llbracket\pdx{^{n-1}u_\sigma}{\vv{n}_F^{n-1}}\right\rrbracket\right\Vert_{L^p(F)}
&\lesssim
\Vert u_\sigma\Vert_{W^{n-1,p}(Q(F))}^{\frac{1}{q}}\Vert u_\sigma\Vert_{W^{n,p}(Q(F))}^{\frac{1}{p}}
\stackrel{\eqref{eqn:localinverseineq}}{\lesssim}
\sigma^{-\frac{1}{p}}\Vert u_\sigma\Vert_{W^{n-1,p}(Q(F))}, \label{eqn:tmp}\\
\left\Vert\left\llbracket\pdx{^{n-1}v_\sigma}{\vv{n}_F^{n-1}}\right\rrbracket\right\Vert_{L^q(F)}
&\lesssim
\Vert v_\sigma\Vert_{W^{n-1,q}(Q(F))}^{\frac{1}{p}}\Vert v_\sigma\Vert_{W^{n,q}(Q(F))}^{\frac{1}{q}}
\stackrel{\eqref{eqn:localinverseineq}}{\lesssim}
\sigma^{-\frac{1}{q}}\Vert u_\sigma\Vert_{W^{n-1,q}(Q(F))}.
\end{align}
The $W^{n,p}$- and $W^{n,q}$-norms in the intermediate step are to be interpreted
in the \enquote{broken}, element-wise sense and $Q(F)$ denotes the two elements
that $F$ is a face of. Thus, using a finite-overlap argument, one obtains together
with another application of the inverse estimates:
\begin{math}
\bigl|\bigl\langle Ju_\sigma, v_\sigma \bigr\rangle\bigr|
\lesssim \sigma^{2n-2}
\Vert u_\sigma\Vert_{W^{n-1,p}(\Omega_\sigma)}
\Vert v_\sigma\Vert_{W^{n-1,q}(\Omega_\sigma)}
\lesssim
\Vert u_\sigma\Vert_{L^p(\Omega_\sigma)}
\Vert v_\sigma\Vert_{L^q(\Omega_\sigma)}.
\end{math}

Let us now consider \eqref{eqn:Jconsistency}. It suffices to establish
this inequality for all integer values $s\in\lbrace 0,1,\ldots,n\rbrace$;
for the intermediate spaces the result then automatically follows by
interpolation. For integers $s\in\lbrace 0,\ldots,n-1\rbrace$ estimate
\eqref{eqn:Jconsistency} follows from
\begin{math}
\bigl|\bigl\langle Ju_\sigma, v_\sigma \bigr\rangle\bigr|
\lesssim \sigma^{n-1}
\Vert u_\sigma\Vert_{W^{n-1,p}(\Omega_\sigma)}
\Vert v_\sigma\Vert_{L^q(\Omega_\sigma)}
\end{math}
by letting $u_\sigma=P_\sigma^n u$ and:
\begin{equation}
\sigma^{n-1}\Vert u_\sigma\Vert_{W^{n-1,p}(\Omega_\sigma)}
\stackrel{\eqref{eqn:globalinverseineqrs}}{\lesssim}
\sigma^{s}\Vert u_\sigma\Vert_{W^{s,p}(\Omega_\sigma)}
\stackrel{\eqref{eqn:Pcontinuity}}{\lesssim}
\sigma^{s}\Vert u\Vert_{W^{s,p}(\Omega_\sigma)}.
\end{equation}
In order to show \eqref{eqn:Jconsistency} for $s=n$, we need to extend $J$'s
domain of definition. For this, note that the derivatives of order $n-1$ of
functions $u\in W^{n,p}(\Omega_\sigma)$ are continuous across hyper-surfaces and
thus $Ju=0$ for such $u$. In other words, $J$ is defined as an operator on
$V_\sigma^{n,p}(\Omega_\sigma)+W^{n,p}(\Omega_\sigma)$ and
$W^{n,p}(\Omega_\sigma)\subset\ker J$. For $\hat{u}:= P_\sigma^n u - u$ equation
\eqref{eqn:tmp} then becomes:
\begin{equation}
\left\Vert\left\llbracket\pdx{^{n-1}\hat{u}}{\vv{n}_F^{n-1}}\right\rrbracket\right\Vert_{L^p(F)}
\lesssim
\Vert\hat{u}\Vert_{W^{n-1,p}(Q(F))}^{\frac{1}{q}}\Vert\hat{u}\Vert_{W^{n,p}(Q(F))}^{\frac{1}{p}},
\end{equation}
where the $W^{n,p}(Q(F))$-norm is again to be interpreted element-wise. Now we
can make use of the approximation properties of $P_\sigma^n$:
\begin{align}
\Vert\hat{u}\Vert_{W^{n-1,p}(Q(F))} = \Vert u -P_\sigma^nu\Vert_{W^{n-1,p}(Q(F))}
&\stackrel{\eqref{eqn:Pconsistency}}{\lesssim}
\sigma\Vert u\Vert_{W^{n,p}(\hat{Q}(F))}, \\
\Vert\hat{u}\Vert_{W^{n,p}(Q(F))} =  \Vert u -P_\sigma^nu\Vert_{W^{n-1,p}(Q(F))}
&\stackrel{\eqref{eqn:Pconsistency}}{\lesssim}
\Vert u\Vert_{W^{n,p}(\hat{Q}(F))}.
\end{align}
Note that the norms on the right do not need to be interpreted element-wise,
because we have assumed $u\in W^{n,p}(\Omega_\sigma)$. Thus
\begin{equation}
\left\Vert\left\llbracket\pdx{^{n-1}\hat{u}}{\vv{n}_F^{n-1}}\right\rrbracket\right\Vert_{L^p(F)}
\lesssim
\sigma^{\frac{1}{q}}\Vert u\Vert_{W^{n,p}(\hat{Q}(F))}.
\end{equation}
Again invoking a finite-overlap argument, one thus obtains:
\begin{equation}
\bigl|\bigl\langle J\hat{u}, v_\sigma \bigr\rangle\bigr| \lesssim
\sigma^n
\Vert u\Vert_{W^{n,p}(\Omega_\sigma)}
\Vert v_\sigma\Vert_{L^q(\Omega_\sigma)}.
\end{equation}
The claim now follows by recalling that $JP_\sigma^nu=J(P_\sigma^nu-u)=J\hat{u}$.
\end{proof}

\subsection{Stability}
The following core result regarding the stability properties of the ghost penalty
operator in $L^2$ has already been established at several places in the literature,
for example by Lehrenfeld~\cite[Lemma~7]{lehrenfeld2018} or Massing et
al.~\cite[Lemma~5.1]{massing2014}:
\begin{lemma}\label{thm:l2stab}
Let $\sigma>0$ be sufficiently small and $\varepsilon>0$ big enough. One then has
for all $u_\sigma\in V_\sigma^n(\Omega_\sigma)$:
\begin{equation}\label{eqn:Jstability}
\Vert u_\sigma\Vert_{L^2(\Omega_\sigma)}^2 \lesssim
\Vert u_\sigma\Vert_{L^2(\Omega_\sigma^\circ)}^2 + \varepsilon\langle Ju_\sigma,u_\sigma\rangle.
\end{equation}
\end{lemma}
From this one easily obtains that $A_\varepsilon^{-1}$ exists and is bounded as
an operator from $V_\sigma^{-n,2}(\Omega_\sigma)\to V_\sigma^{n,2}(\Omega_\sigma)$.
We will now establish that $A_\varepsilon^{-1}$ also is bounded as an
operator from $V_{\sigma}^{-n,p}(\Omega_\sigma)\to V_\sigma^{n,p}(\Omega_\sigma)$,
$1\leq p\leq\infty$.

\begin{lemma}\label{thm:stability}
For $\sigma>0$ small enough and $\varepsilon>0$ sufficiently large, the approximate
extension operator $A_\varepsilon^{-1}$ is bounded. In other words, for all $f\in
V_\sigma^{-n,p}(\Omega_\sigma)$, $1\leq p\leq\infty$ it holds that:
\begin{equation}\label{eqn:stability}
\Vert A_\varepsilon^{-1}f\Vert_{L^p(\Omega_\sigma)}\lesssim
\Vert f\Vert_{V_\sigma^{-n,p}(\Omega_\sigma)}.
\end{equation}
\end{lemma}
\begin{proof}

Our proof is similar to those of Crouzeix and Thom\'ee~\cite{crouzeix1987} as
well as Douglas, Dupont, and Wahlbin.~\cite{douglas1974} Because of
\cref{thm:equivalence}, it suffices to consider functionals of the form
$\idx{\Omega_\sigma}{}{fv_\sigma}{x}$, $f\in L^p(\Omega_\sigma)$. We fix an
arbitrary $\vv{j}\in\mathbb{Z}^\spdim$ such that $Q_{\vv{j}}\in\Omega_\sigma$.
We set $f_{\vv{j}} = f$ on $Q_\vv{j}$ and  $f_{\vv{j}}\equiv 0$ else and define
$u_{\sigma,\vv{j}}:=A_\varepsilon^{-1}f_{\vv{j}}$. We will show that
$u_{\sigma,\vv{j}}$ decays at an exponential rate away from $Q_{\vv{j}}$. To
this end we define the domains $D_{\vv{j},0}:=\emptyset$, $D_{\vv{j},1}:=Q_{\vv{j}}$,
and $D_{\vv{j},k}:=\bigl\lbrace Q_{\vv{i}}\in\Omega_\sigma\,\bigl|\bigr.\,|\vv{i}-
\vv{j}|<k\bigr\rbrace$ for all other integers $k$, where $|\cdot|$ denotes the
max-norm over $\mathbb{Z}^\spdim$. Furthermore, we set:
\begin{equation}
\mathcal{F}_{\sigma}^{\vv{j},k}:=
\bigl\lbrace
F\in\mathcal{F}_\sigma\,\bigl|\bigr.\,\text{
Both elements that $F$ is a face of are in $\Omega_\sigma\setminus D_{\vv{j},k}$}
\bigr\rbrace.
\end{equation}
First we note that because $f_{\vv{j}}\equiv 0$ outside of $Q_{\vv{j}}$, one has
by the definition of $u_{\sigma,\vv{j}}$ that $\langle A_\varepsilon u_{\sigma,\vv{j}},
v_\sigma\rangle = \langle f_{\vv{j}},v_\sigma\rangle = 0$ for all $v_\sigma\in V_\sigma^n
(\Omega_\sigma)$ that vanish on $Q_{\vv{j}}$. We now choose such a special $v_\sigma$.
Let $k\geq n$ and set the B-spline coefficients of $v_\sigma$ such that $v_\sigma=
u_{\sigma,\vv{j}}$ on $\Omega_\sigma\setminus D_{\vv{j},k}$ and set the remaining
coefficients to zero. It follows that $v_\sigma\equiv 0$ on $D_{\vv{j},k-(n-1)}$.
Because $\langle A_\varepsilon u_{\sigma,\vv{j}},v_\sigma\rangle = 0$ one
easily obtains that:
\begin{multline}
\idx{\Omega\setminus D_{\vv{j},k}}{}{u_{\sigma,\vv{j}}^2}{\vv{x}} +
\varepsilon\sigma^{2n-1}\sum_{F\in\mathcal{F}_{\sigma}^{\vv{j},k}}
\idx{F}{}{\left\llbracket
\pdx{^{n-1}u_{\sigma,\vv{j}}}{\vv{n}_{F}^{n-1}}\right\rrbracket^2}{S} = \\
-\left(
\idx{\Omega\cap D_{\vv{j},k}}{}{u_{\sigma,\vv{j}}v_\sigma}{\vv{x}} +
\varepsilon\sigma^{2n-1}
\sum_{F\in\mathcal{F}_\sigma\setminus\mathcal{F}_{\sigma}^{\vv{j},k}}
\idx{F}{}{
\left\llbracket \pdx{^{n-1}u_{\sigma,\vv{j}}}{\vv{n}_{F}^{n-1}}\right\rrbracket
\left\llbracket \pdx{^{n-1}v_\sigma}{\vv{n}_{F}^{n-1}}\right\rrbracket}{S}
\right).
\end{multline}
Because of \cref{thm:l2stab} the left side of this equality can be bounded
from below by $\Vert u_{\sigma,\vv{j}}\Vert_{L^2(\Omega_\sigma\setminus
D_{\vv{j},k})}^2$. Because $v_\sigma\equiv 0$ on $D_{\vv{j},k-(n-1)}$, the
integral on the right can be upper bounded by $\Vert u_{\sigma,\vv{j}}
\Vert_{L^2(D_{\vv{j},k}\setminus D_{\vv{j},k-(n-1)})}\Vert v_{\sigma}\Vert_{%
L^2(D_{\vv{j},k}\setminus D_{\vv{j},k-(n-1)})}$. The same bound follows for the
sum, using the arguments in the proof of \cref{thm:Jcontinuity}. In fact, for
most choices of $\vv{j}$ and $k$, this sum is empty. But clearly, by the stability
of the B-spline basis, we have $\Vert v_{\sigma}\Vert_{L^2(D_{\vv{j},k}
\setminus D_{\vv{j},k-(n-1)})}\lesssim \Vert u_{\sigma,\vv{j}}\Vert_{L^2(D_{\vv{j},k}
\setminus D_{\vv{j},k-(n-1)})}$. Thus, in total we obtain the existence of a
constant $C>0$ such that:
\begin{multline}
\Vert u_{\sigma,\vv{j}}\Vert_{L^2(\Omega_\sigma\setminus D_{\vv{j},k})}^2
\leq 
C\Vert u_{\sigma,\vv{j}}\Vert_{L^2(D_{\vv{j},k}\setminus D_{\vv{j},k-(n-1)})}^2 \\
=C\left(
\Vert u_{\sigma,\vv{j}}\Vert_{L^2(\Omega_\sigma\setminus D_{\vv{j},k-(n-1)})}^2 -
\Vert u_{\sigma,\vv{j}}\Vert_{L^2(\Omega_\sigma\setminus D_{\vv{j},k})}^2
\right),
\end{multline}
and therefore:
\begin{equation}
\Vert u_{\sigma,\vv{j}}\Vert_{L^2(\Omega_\sigma\setminus D_{\vv{j},k})}^2
\leq\frac{C}{1+C}
\Vert u_{\sigma,\vv{j}}\Vert_{L^2(\Omega_\sigma\setminus D_{\vv{j},k-(n-1)})}^2.
\end{equation}
For large values of $k$ this argument can now be repeated on the right hand
side, leading to the existence of another constant $0<\gamma<1$ such that
\begin{math}
\Vert u_{\sigma,\vv{j}}\Vert_{L^2(\Omega_\sigma\setminus D_{\vv{j},k})}^2
\lesssim \gamma^{2k} \Vert u_{\sigma,\vv{j}}\Vert_{L^2(\Omega_\sigma)}^2.
\end{math}
This is the desired exponential decay. Using \cref{thm:l2stab}, we get together
with the inverse estimates:
\begin{multline}
\Vert u_{\sigma,\vv{j}}\Vert_{L^2(\Omega_\sigma)}^2
\stackrel{\cref{eqn:Jstability}}{\lesssim}
\langle A_\varepsilon u_{\sigma,\vv{j}},u_{\sigma,\vv{j}}\rangle =
\idx{Q_{\vv{j}}}{}{fu_{\sigma,\vv{j}}}{x} \leq
\Vert f\Vert_{L^p(Q_{\vv{j}})}\Vert u_{\sigma,\vv{j}}\Vert_{L^q(Q_{\vv{j}})} \\
\stackrel{\cref{eqn:localinverseineq}}{\lesssim}\sigma^{\frac{\spdim}{2}-\frac{\spdim}{p}}
\Vert f\Vert_{L^p(Q_{\vv{j}})}\Vert u_{\sigma,\vv{j}}\Vert_{L^2(Q_{\vv{j}})}
\leq\sigma^{\frac{\spdim}{2}-\frac{\spdim}{p}}
\Vert f\Vert_{L^p(Q_{\vv{j}})}\Vert u_{\sigma,\vv{j}}\Vert_{L^2(\Omega_\sigma)},
\end{multline}
and thus $\Vert u_{\sigma,\vv{j}}\Vert_{L^2(\Omega_\sigma)}\lesssim
\sigma^{\frac{\spdim}{2}-\frac{\spdim}{p}}\Vert f\Vert_{L^p(Q_{\vv{j}})}$. For
every $\vv{i}\in\mathbb{Z}^\spdim$, $Q_{\vv{i}}\in\Omega_\sigma$ this leads to:
\begin{equation}
\Vert u_{\sigma,\vv{j}}\Vert_{L^p(Q_{\vv{i}})}
\stackrel{\cref{eqn:localinverseineq}}{\lesssim}\sigma^{\frac{\spdim}{p}-\frac{\spdim}{2}}
\Vert u_{\sigma,\vv{j}}\Vert_{L^2(Q_{\vv{i}})}
\lesssim \sigma^{\frac{\spdim}{p}-\frac{\spdim}{2}}
\gamma^{\vert\vv{i}-\vv{j}\vert}\Vert u_{\sigma,\vv{j}}\Vert_{L^2(\Omega_\sigma)}
\lesssim
\gamma^{\vert\vv{i}-\vv{j}\vert}\Vert f\Vert_{L^p(Q_{\vv{j}})}.
\end{equation}

Let consider the case $p=\infty$. Noting that $u_\sigma:= A_\varepsilon^{-1}f
= \sum_{\vv{j}}u_{\sigma,\vv{j}}$ we obtain by the triangular inequality for
arbitrary $\vv{i}$:
\begin{equation}
\Vert u_\sigma\Vert_{L^\infty(Q_{\vv{i}})} \lesssim
\Vert f\Vert_{L^\infty(\Omega_\sigma)}\sum_{\vv{j}}\gamma^{|\vv{i}-\vv{j}|}.
\end{equation}
Because of the grid's uniformity and the exponential decay, the latter sum
remains bounded for any $\vv{i}$, and therefore $\Vert u_\sigma\Vert_{L^\infty(\Omega_\sigma)}
\lesssim \Vert f\Vert_{L^\infty(\Omega_\sigma)}$. For $p=1$ we obtain similarly:
\begin{equation}
\Vert u_\sigma\Vert_{L^1(\Omega_\sigma)} =
\sum_{\vv{i}}\Vert u_\sigma\Vert_{L^1(Q_{\vv{i}})} \lesssim
\sum_{\vv{i,j}}\gamma^{\vert\vv{i}-\vv{j}\vert}\Vert f\Vert_{L^1(Q_{\vv{j}})}
\stackrel{\text{H\"older}}{\leq}\Vert f\Vert_{L^1(\Omega_\sigma)}
\max_{\vv{j}}\sum_{\vv{i}}\gamma^{\vert\vv{i}-\vv{j}\vert}
\end{equation}
and thus $\Vert u_\sigma\Vert_{L^1(\Omega_\sigma)}\lesssim
\Vert f\Vert_{L^1(\Omega_\sigma)}$. For $1<p<\infty$ the result now follows by the
Riesz--Thorin interpolation theorem.
\end{proof}

\subsection{Condition Numbers}
In order to implement the approximate extension operator in practice, it is
important that the condition number of the corresponding system matrix
$\sm{A}_\varepsilon$ remains bounded. Let us abbreviate $N:=\dim V_\sigma^{n,p}
(\Omega_\sigma)$. We may assign a numbering $1,\ldots,N$ to the index set
$\Lambda_{\sigma}^{n}(\Omega_\sigma)$ and refer to the B-splines $b_{\sigma,
\gv{\lambda}}^n$ as $b_i$, $i\in\lbrace 1,\ldots, N\rbrace$. The system matrix
$\sm{A}_\varepsilon\in\mathbb{R}^{N\times N}$ is then defined via:
\begin{equation}
\sm{e}_j^\top\sm{A}_{\varepsilon}\sm{e}_i = \langle A_\varepsilon b_i, b_j\rangle
\qquad \forall i,j\in\lbrace 1,\ldots, N\rbrace,
\end{equation}
where $\sm{e}_i, \sm{e}_j\in\mathbb{R}^N$ refer to the $i$th and $j$th
Cartesian basis vectors, respectively. One easily obtains the following
corollary, which guarantees that systems involving $\sm{A}_\varepsilon$
can efficiently be solved using iterative solvers. 
\begin{corollary}[Condition of $\sm{A}_\varepsilon$]
The system matrix $\sm{A}_\varepsilon\in\mathbb{R}^{N\times N}$ is symmetric
$\sm{A}_\varepsilon = \sm{A}_\varepsilon^\top$, positive definite:
\begin{equation}
\sm{u}_\sigma^\top\sm{A}_\varepsilon\sm{u}_\sigma
\gtrsim
\sigma^{\spdim}\Vert\sm{u}_\sigma\Vert_{\ell^2}^2
\qquad\forall\sm{u}_\sigma\in\mathbb{R}^N,
\end{equation}
and well-conditioned:
\begin{multline}
\forall\sm{u}_\sigma\in\mathbb{R}^N:\ 
\sigma^\spdim\Vert\sm{u}_\sigma\Vert_{\ell^p}
\lesssim
\Vert\sm{A}_{\varepsilon}\sm{u}_\sigma\Vert_{\ell^p}
\lesssim
\sigma^\spdim\Vert\sm{u}_\sigma\Vert_{\ell^p} \\
\Longrightarrow
\cond_p(\sm{A}_{\varepsilon}) =
\Vert\sm{A}_{\varepsilon}\Vert_{\ell^p\to\ell^p}
\Vert \sm{A}_{\varepsilon}^{-1}\Vert_{\ell^p\to\ell^p} \sim 1.
\end{multline}
\end{corollary}
\begin{proof}
The symmetry of $\sm{A}_\varepsilon$ is obvious. With every $\sm{u}_\sigma\in\mathbb{R}^N$
we associate $u_\sigma = \sum_{i=1}^{N}u_ib_i$. Then, with help of the stability of the
B-spline basis:
\begin{equation}
\sm{u}_\sigma^\top\sm{A}_\varepsilon\sm{u}_\sigma =
\langle A_\varepsilon u_\sigma, u_\sigma\rangle
\stackrel{\eqref{eqn:Jstability}}{\gtrsim}
\Vert u_\sigma\Vert_{L^2(\Omega_\sigma)}^2
\stackrel{\eqref{eqn:Bstab}}{\gtrsim}
\sigma^{\spdim}\Vert\sm{u}_\sigma\Vert_{\ell^2}^2.
\end{equation}
Moreover, for the lower inequality:
\begin{multline}
\sigma^\spdim\Vert\sm{u}_\sigma\Vert_{\ell^p}
\stackrel{\eqref{eqn:Bstab}}{\lesssim}
\sigma^{\frac{\spdim}{q}}\Vert u_\sigma\Vert_{L^p(\Omega_\sigma)}
\stackrel{\eqref{eqn:stability}}{\lesssim}
\sigma^{\frac{\spdim}{q}}\Vert A_\varepsilon u_\sigma\Vert_{V^{-n,p}(\Omega_\sigma)} \\
= \sigma^{\frac{\spdim}{q}}\sup_{v_\sigma\in V_{\sigma}^{n,q}(\Omega_\sigma)}
\frac{\langle A_\varepsilon u_\sigma, v_\sigma\rangle}{\Vert v_\sigma\Vert_{L^q(\Omega)}}
\stackrel{\eqref{eqn:Bstab}}{\lesssim}
\sigma^{\frac{\spdim}{q}}\sup_{v_\sigma\in V_{\sigma}^{n,q}(\Omega_\sigma)}
\frac{\sm{v}_\sigma^\top \sm{A}_\varepsilon\sm{u}_\sigma}{\sigma^{\frac{\spdim}{q}}\Vert\sm{v}_\sigma\Vert_{\ell^q}}
=
\Vert\sm{A}_\varepsilon\sm{u}_\sigma\Vert_{\ell^p}.
\end{multline}
Similarly, for the upper inequality:
\begin{multline}
\Vert\sm{A}_\varepsilon\sm{u}_\sigma\Vert_{\ell^p} =
\sup_{\sm{v}_\sigma\in\mathbb{R}^N}
\frac{\sm{v}_\sigma^\top\sm{A}_\varepsilon\sm{u}_\sigma}{\Vert\sm{v}_\sigma\Vert_{\ell^q}}
\stackrel{\eqref{eqn:Bstab}}{\lesssim}
\sup_{\sm{v}_\sigma\in\mathbb{R}^N}\frac{\langle A_\varepsilon u_\sigma,v_\sigma\rangle}%
{\sigma^{-\frac{\spdim}{q}}\Vert v_\sigma\Vert_{L^q(\Omega_\sigma)}} \\
\stackrel{\eqref{eqn:Jcontinuity}}{\lesssim}
\sigma^{\frac{\spdim}{q}}\Vert u_\sigma\Vert_{L^p(\Omega_\sigma)}
\stackrel{\eqref{eqn:Bstab}}{\lesssim}\sigma^{\spdim}\Vert\sm{u}_\sigma\Vert_{\ell^p}.
\end{multline}
\end{proof}

\subsection{Convergence}
Every $u_\Omega\in L^p(\Omega)$ may be interpreted as an element of
$V_\sigma^{-n,p}(\Omega_\sigma)$ by setting 
\begin{equation}
\langle u_\Omega, v_\sigma \rangle := \idx{\Omega}{}{u_\Omega v_\sigma}{\vv{x}} \qquad
\forall v_\sigma\in V_\sigma^{n,q}(\Omega_\sigma).
\end{equation}
Similarly, any element of $W^{-s,p}(\Omega)$, $0 < s < n - \tfrac{1}{p}$, can
be interpreted as an element of $V_\sigma^{-n,p}(\Omega_\sigma)$ by restricting
the test functions from $V_\sigma^{n,q}(\Omega_\sigma)\subset W^{s,q}(\Omega_\sigma)$
to $\Omega$. We now prove that $A_\varepsilon^{-1}u_\Omega$ converges to the Stein
extension on the entire fictitious domain $\Omega_\sigma$ at an optimal rate.
\begin{theorem}[Approximate Extension]\label{thm:convergence}
Let $n\in\mathbb{N}$, $k,l\in\mathbb{N}_0$, $u_\Omega\in W^{k,p}(\Omega)$,
$0\leq k\leq n$, $1\leq p\leq\infty$. Let $\sigma>0$ be sufficiently small
and $\varepsilon>0$ big enough. Then the approximate extension operator
$A_\varepsilon^{-1}$ fulfills:
\begin{align}\label{eqn:Aconsistency}
\Vert Eu_\Omega - A_\varepsilon^{-1}u_\Omega\Vert_{W^{l,p}(\Omega_\sigma)}
&\lesssim \sigma^{k-l}\Vert u_\Omega\Vert_{W^{k,p}(\Omega)} &
0&\leq l\leq\min\lbrace k,n-1\rbrace, \\
\Vert A_\varepsilon^{-1}u_\Omega\Vert_{W^{l,p}(\Omega_\sigma)}
&\lesssim \sigma^{k-l}\Vert u_\Omega\Vert_{W^{k,p}(\Omega)} &
k&\leq l\leq n-1.\label{eqn:Acontinuity}
\end{align}
The hidden constant is independent of $\sigma$, $u_\Omega$, and how $\pd\Omega$
intersects the Cartesian grid. If one interprets the norms on the left side of
the inequalities in the broken, element-wise sense, they also remain true for
$l=n$.
\end{theorem}
\begin{proof}
Let us first consider \eqref{eqn:Aconsistency} and note that
\begin{equation}
\Vert Eu_\Omega-A_\varepsilon^{-1}u_\Omega\Vert_{W^{l,p}(\Omega_\sigma)} \leq
\Vert Eu_\Omega-PEu_\Omega\Vert_{W^{l,p}(\Omega_\sigma)} +
\Vert PEu_\Omega-A_\varepsilon^{-1}u_\Omega\Vert_{W^{l,p}(\Omega_\sigma)},
\end{equation}
where we abbreviated $P=P_\sigma^n$. The first term can be bounded as desired
by \eqref{eqn:Pconsistency} and the continuity of the Stein extension. For the
second term it suffices to consider the case $l = 0$, the remaining cases then
follow by the inverse estimates \eqref{eqn:globalinverseineqrs}. Thus:
\begin{multline}
\Vert PEu_\Omega-A_\varepsilon^{-1}u_\Omega\Vert_{L^p(\Omega_\sigma)} =
\bigl\Vert A_\varepsilon^{-1}A_\varepsilon\bigl(PEu_\Omega-A_\varepsilon^{-1}u_\Omega\bigr)\bigr\Vert_{L^p(\Omega_\sigma)}
\stackrel{\eqref{eqn:stability}}{\lesssim} \\
\bigl\Vert A_\varepsilon\bigl(PEu_\Omega-A_\varepsilon^{-1}u_\Omega\bigr)\bigr\Vert_{V^{-n,p}(\Omega_\sigma)} =
\bigl\Vert A_\varepsilon PEu_\Omega-u_\Omega\bigr\Vert_{V^{-n,p}(\Omega_\sigma)}.
\end{multline}
For this last term, we obtain for arbitrary $v_\sigma\in V_\sigma^{n,q}(\Omega_\sigma)$:
\begin{equation}
\langle A_\varepsilon PEu_\Omega-u_\Omega, v_\sigma\rangle =
\idx{\Omega}{}{\bigl(PEu_\Omega-u_\Omega\bigr)v_\sigma}{\vv{x}} +
\varepsilon\langle JPEu_\Omega,v_\sigma\rangle.
\end{equation}
With the help of Hölder's inequality, \cref{eqn:Pconsistency}, and the
boundedness of the  Stein extension operator, the integral can be bounded by
$\sigma^{k}\Vert u_\Omega\Vert_{W^{k,p}(\Omega)}\Vert v_\sigma\Vert_{L^q(\Omega_\sigma)}$.
The same bound follows for the second term by \cref{eqn:Jconsistency}. Thus
$\bigl\Vert A_\varepsilon PEu_\Omega-u_\Omega\bigr\Vert_{V^{-n,p}(\Omega_\sigma)}
\lesssim\sigma^k\Vert u_\Omega\Vert_{W^{n,p}(\Omega)}$ as desired. For
\eqref{eqn:Acontinuity} we now obtain:
\begin{multline}
\Vert A_\varepsilon^{-1}u_\Omega\Vert_{W^{l,p}(\Omega_\sigma)}
\stackrel{\eqref{eqn:globalinverseineqrs}}{\lesssim}
\sigma^{k-l}\Vert A_\varepsilon^{-1}u_\Omega\Vert_{W^{k,p}(\Omega_\sigma)} \\
\leq \sigma^{k-l} \bigl(\Vert A_\varepsilon^{-1}u_\Omega-Eu_\Omega\Vert_{W^{k,p}(\Omega_\sigma)}
+ \Vert Eu_\Omega\Vert_{W^{k,p}(\Omega_\sigma)}\bigr),
\end{multline}
where the first term can now be bounded as desired by \eqref{eqn:Aconsistency}
and the second by the continuity of the Stein extension.
\end{proof}
Every function in $V_\sigma^n(\qdom)$ can be extended to $\wholespace$ by
simply removing the restriction on the B-splines it is composed of. Because of
\eqref{eqn:Bstab}, one also has $\Vert A_\varepsilon^{-1}u_\Omega\Vert_{L^p
(\wholespace)}\lesssim \Vert u_\Omega\Vert_{L^p(\Omega)}$. When considered only
on the domain $\Omega$, on the other hand, we also obtain the following super-%
convergence result.
\begin{corollary}[Super-Convergence]\label{thm:superconvergence}
Under the same conditions as the previous theorem we have
for all $l\in\mathbb{Z}$, $-n\leq l\leq \min\lbrace k,n-1\rbrace$:
\begin{equation}
\Vert u_\Omega - A_\varepsilon^{-1}u_\Omega\Vert_{W^{l,p}(\Omega)}
\lesssim
\sigma^{k-l}\Vert u_\Omega\Vert_{W^{k,p}(\Omega)},
\end{equation}
The hidden constant is independent of $\sigma$, $u_\Omega$, and how $\pd\Omega$ 
intersects the Cartesian grid. If one interprets the norm on the left in the
broken, element-wise sense, the statement also remains true for $l=n$.
\end{corollary}
\begin{proof}
For non-negative $l$, this result is obtained from \eqref{eqn:Aconsistency} by
restriction from $\Omega_\sigma$ to $\Omega$. Let us thus consider $l<0$ and
denote $u_\sigma := A_\varepsilon^{-1}u_\Omega$. Then, for all $\varphi\in W^{-l,q}(\Omega)$
\begin{equation}
\idx{\Omega}{}{\bigl(u_\Omega - u_\sigma\bigr)\varphi}{\vv{x}} = 
\underbrace{\idx{\Omega}{}{\bigl(u_\Omega - u_\sigma\bigr)PE\varphi}{\vv{x}}}_{=\varepsilon\langle Ju_\sigma,PE\varphi\rangle}  +
\idx{\Omega}{}{\bigl(u_\Omega - u_\sigma\bigr)\bigl(\varphi-PE\varphi\bigr)}{\vv{x}}.
\end{equation}
The second term can be bounded as desired by Hölder's inequality, \eqref{eqn:Pconsistency},
and \eqref{eqn:Aconsistency}:
\begin{multline}
\idx{\Omega}{}{\bigl(u_\Omega - u_\sigma\bigr)\bigl(\varphi-PE\varphi\bigr)}{\vv{x}} \leq
\Vert u_\Omega - A_\varepsilon^{-1}u_\Omega\Vert_{L^p(\Omega)}
\Vert \varphi-PE\varphi\Vert_{L^q(\Omega)} \\
\lesssim \sigma^{k-l}
\Vert u_\Omega\Vert_{W^{k,p}(\Omega)}
\Vert\varphi\Vert_{W^{-l,q}(\Omega)}.
\end{multline}
For the first term note that $J$ is symmetric:
$\langle Ju_\sigma,PE\varphi\rangle = \langle JPE\varphi,u_\sigma\rangle$. We
therefore obtain using the same arguments as in the proof of \cref{thm:Jcontinuity}:
\begin{equation}
\langle JPE\varphi,u_\sigma\rangle \lesssim
\sigma^{k-l}\Vert PE\varphi\Vert_{W^{-l,q}(\Omega_\sigma)}\Vert u_\sigma\Vert_{W^{k,p}(\Omega_\sigma)},
\end{equation}
where for $k=n$ the $W^{k,p}$-norm on the right is to be interpreted in the
\enquote{broken}, element-wise sense. The claim now follows by applying
\eqref{eqn:Acontinuity}, \eqref{eqn:Pcontinuity}, and the continuity of the
Stein extension operator.
\end{proof}

By interpolation these results also extend to the intermediate spaces. The
conditions become slightly technical when interpolating on $k$ and $l$
simultaneously, however. On the other hand by interpolating on only one of
them, one for example immediately obtains:
\begin{align}
\Vert A_\varepsilon^{-1}u_\Omega - Eu_\Omega\Vert_{L^p(\Omega_\sigma)}
&\lesssim \sigma^{s}\Vert u_\Omega\Vert_{W^{s,p}(\Omega)} &  0&\leq s\leq n, \\
\Vert A_\varepsilon^{-1}u_\Omega - u_\Omega\Vert_{W^{s,p}(\Omega)}
&\lesssim \sigma^{n-s}\Vert u_\Omega\Vert_{W^{n,p}(\Omega)} & -n&\leq s\leq n-1.
\end{align}

\subsection{Application to Particle Fields}
Our aim is to apply the approximate extension operator $A_\varepsilon^{-1}$ to
an evolving particle field $W^{-s,p}(\Omega)\ni u_h(t)\approx u(t)\in
W^{s,p}(\Omega)$. For this, we consider the following particle method:
given $n\in\mathbb{N}$, $n>\spdim$, and $\sigma>0$ we will set $h=2^{-k}\sigma$,
$k\in\mathbb{N}_0$, $m=n$, such that $V_\sigma^n(\Omega)\subset V_h^m(\Omega)$.
Given $u_0\in W^{s,p}(\Omega)\cap L^\infty(\Omega)$, $0\leq s\leq n$, $1<p
\leq\infty$, we set $\tilde{u}_{0,h}=\tilde{u}_{0,\sigma}=A_\varepsilon^{-1}u_0$.
The particle approximation $u_{0,h}$ is then constructed from $\tilde{u}_{0,h}$
as described in \cref{subsec:construction}. Finally, $u_h(t)$ is defined by
modifying the particle positions $\vv{x}_i$, $i=1,\ldots,N$ according to the
system of ODEs:
\begin{equation}
\ddx{\vv{x}_i}{t}(t) = \vv{a}(\vv{x}_i(t),t)\qquad i=1,\ldots,N.
\end{equation}
We then obtain the following estimate for the error $\Vert Eu(t)-
A_\varepsilon^{-1}u_h(t)\Vert_{L^p(\Omega_\sigma)}$, which is the main result
of this article.

\begin{theorem}\label{thm:main}
Let $u_0\in W^{s,p}(\Omega)\cap L^{\infty}(\Omega)$, $0\leq s\leq n$, $1<p\leq
\infty$, $n\in\mathbb{N}$, $n > \spdim$, and let the given velocity field $\vv{a}
\in L^\infty\bigl(W^{n,\infty}(\Omega),[0,T]\bigr)$ be sufficiently smooth. Let
the particle approximation $u_h(t)$ be defined as described above. Then for every
$t\in [0,T]$ and for arbitrarily small $\delta>0$ the regularized particle field
$A_\varepsilon^{-1}u_h(t)$ fulfills the following error bound:
\begin{equation}
\Vert Eu(t) - A_\varepsilon^{-1}u_h(t)\Vert_{L^p(\Omega_\sigma)}
\lesssim
\sigma^s                                 \Vert u_0\Vert_{W^{s,p}(\Omega)} +
\left(\frac{h}{\sigma}\right)^{n-\delta} \Vert u_0\Vert_{L^\infty(\Omega)}.
\end{equation}
Moreover, if $s=k$ is an integer, one has for all integers $0\leq l
\leq\min\{k,n-1\}$:
\begin{equation}
\Vert Eu(t) - A_\varepsilon^{-1}u_h(t)\Vert_{W^{l,p}(\Omega_\sigma)}
\lesssim
\sigma^{k-l}                                        \Vert u_0\Vert_{W^{k,p}(\Omega)} +
\sigma^{-l}\left(\frac{h}{\sigma}\right)^{n-\delta} \Vert u_0\Vert_{L^\infty(\Omega)}.
\end{equation}
\end{theorem}
\begin{proof}
Let us denote by $u(t)$ and $\tilde{u}_h(t)$ the respective exact solutions of
the advection equation with initial data $u_0$ and $\tilde{u}_{0,h}$. We can split
the error into three parts:
\begin{multline}
\Vert Eu(t)-A_\varepsilon^{-1}u_h(t)\Vert_{W^{l,p}(\Omega_\sigma)} \leq
\overbrace{\Vert Eu(t)-A_\varepsilon^{-1}u(t)\Vert_{W^{l,p}(\Omega_\sigma)}}^{\mathrm{(I)}} + \\
\underbrace{\Vert A_\varepsilon^{-1}\bigl(u(t)-\tilde{u}_h(t)\bigr)\Vert_{W^{l,p}(\Omega_\sigma)}}_{\mathrm{(II)}} +
\underbrace{\Vert A_\varepsilon^{-1}\bigl(\tilde{u}_h(t)-u_h(t)\bigr)\Vert_{W^{l,p}(\Omega_\sigma)}}_{\mathrm{(III)}}.
\end{multline}

By \cref{thm:convergence} and the stability of the advection equation~%
\eqref{eqn:stableadvection} we have $\mathrm{(I)}\lesssim\sigma^{s}\Vert
u_0\Vert_{W^{s,p}(\Omega)}$ for $l=0$ and respectively $\mathrm{(I)}\lesssim
\sigma^{k-l}\Vert u_0\Vert_{W^{k,p}(\Omega)}$ otherwise. For the terms
$\mathrm{(II)}$ and $\mathrm{(III)}$ it suffices to consider the case $l=0$, the
other cases follow by the inverse estimate \cref{eqn:globalinverseineqrs}.
For $\mathrm{(II)}$ we first make use of \cref{thm:stability} to obtain
$\mathrm{(II)}\lesssim\Vert u(t)-\tilde{u}_h(t)\Vert_{V^{-n,p}(\Omega_\sigma)}$. 
By H\"older's inequality we see that $\Vert u(t)-\tilde{u}_h(t)
\Vert_{V^{-n,p}(\Omega_\sigma)} \leq \Vert u(t)-\tilde{u}_h(t)
\Vert_{L^p(\Omega)}$ and subsequently obtain by the same arguments as for
the first term that: $\mathrm{(II)}\lesssim \sigma^s\Vert u_0
\Vert_{W^{s,p}(\Omega)}$.

For the last term we denote $r:=n-\delta$ and note that because $r\geq\spdim$
we have $u_h\in W^{-r,\infty}(\Omega)$. Furthermore, one trivially has
$\Vert\cdot\Vert_{L^p(\Omega_\sigma)} \lesssim \Vert\cdot\Vert_{L^\infty(\Omega_\sigma)}$.
Thus
\begin{equation}
\mathrm{(III)}\lesssim
\Vert A_\varepsilon^{-1}\bigl(\tilde{u}_h(t)-u_h(t)\bigr)\Vert_{L^\infty(\Omega_\sigma)}
\stackrel{\cref{eqn:stability}}{\lesssim}
\Vert \tilde{u}_h(t)-u_h(t)\Vert_{V^{-n,\infty}(\Omega_\sigma)}.
\end{equation}
At this point we make use of the fact that we have
$V_\sigma^n(\Omega_\sigma)\subset W^{r,1}(\Omega_\sigma)$, $n-1<r<n$; the reason
why we introduced fractional order Sobolev spaces. By inverse estimates one obtains:
\begin{multline}
\Vert\tilde{u}_h(t)-u_h(t)\Vert_{V^{-n,\infty}(\Omega_\sigma)}
= \sup_{v_\sigma\in V^{n,1}(\Omega_\sigma)}
\langle \tilde{u}_h(t)-u_h(t),v_\sigma\rangle / \Vert v_\sigma\Vert_{L^1(\Omega_\sigma)}\\
\stackrel{\cref{eqn:globalinverseineqrs}}{\lesssim}
 \sigma^{-r}{\kern -1em}\sup_{v_\sigma\in V^{n,1}(\Omega_\sigma)}
\langle\tilde{u}_h(t)-u_h(t),v_\sigma\rangle/\Vert v_\sigma\Vert_{W^{r,1}(\Omega_\sigma)}
\leq \sigma^{-r} \Vert \tilde{u}_h(t)-u_h(t)\Vert_{W^{-r,\infty}(\Omega)}.
\end{multline}
Now, by the stability of the advection equation~\eqref{eqn:stableadvection} and
\cref{thm:particleerror}:
\begin{multline}
\sigma^{-r}\Vert\tilde{u}_h(t)-u_h(t)\Vert_{W^{-r,\infty}(\Omega)}  \lesssim
\sigma^{-r}\Vert\tilde{u}_{0,h}-u_{0,h}\Vert_{W^{-r,\infty}(\Omega)} \\ \lesssim
\biggl(\frac{h}{\sigma}\biggr)^r\Vert\tilde{u}_{0,h}\Vert_{L^\infty(\Omega_\sigma)}
\lesssim \biggl(\frac{h}{\sigma}\biggr)^r\Vert u_0\Vert_{L^\infty(\Omega)}.
\end{multline}
\end{proof}

When restricted to the domain $\Omega$, it is a simple task to confirm that this
result also holds for negative $l$, analogous to the super-convergence result 
\cref{thm:superconvergence}. If one assumes that the exact solution is smooth,
these results suggest choosing $h\sim\sigma^2$ in order to balance the error
contributions, similar to the earliest analyses.~\cite{hald1979} In that case
this choice in particular implies that one essentially has (up to $\delta$),
$\Vert u(t)-A_\varepsilon^{-1}u_h(t)\Vert_{W^{-n,p}(\Omega)}=
\mathcal{O}(\sigma^{2n})=\mathcal{O}(h^n)$. In other words
$A_\varepsilon^{-1} u_h(t)$ and $u_h(t)$ asymptotically fulfill the same error
bound which is the most one can expect from a regularization scheme.

\section{Discussion and Outlook}\label{sec:outlook}
In general, it is inherently difficult to choose $h$ such that the error
contributions from regularization and quadrature are balanced. In particular,
one usually does not know a-priori how smooth the solution actually is. Let us
first consider the choice $h=\sigma$. Clearly, upon initialization, we have
$\Vert\tilde{u}_{h,0}-u_{h,0}\Vert_{V_\sigma^{-n,p}(\Omega_\sigma)} = 0$, and it
is unlikely that for small times $0<t\ll T$ this error immediately increases to
significant levels. On the other hand, it is well-known from computational
practice that this choice of $h$ does not lead to converging schemes for
extended periods of time. After all, the advection equation is stable in
$W^{-s,p}$- and not in $V_\sigma^{-n,p}$-norms. This motivates so-called
\emph{remeshed particle methods}, where the particle field is reinitialized
with its regularized version after every other time-step or so. Practice has
shown that these methods seem to work well.

On the other hand, the choice $h\sim\sigma^2$ requires one to manage
significantly larger numbers of particles which at early times $t$ do not
significantly improve the method's accuracy. But there also is an advantage to
this approach: such a particle field carries sub-$\sigma$-scale information
about small features, which can arise over time due to the distortion of $u_0$
by the velocity field. Furthermore, in a computer implementation it is easy to
handle large numbers of particles, as there is no connectivity involved. A
reinitialization of the particle field destroys this sub-grid information.

In practice, particle fields tend to get thinned out in some parts of the domain,
and clustered in others. In fact, being an exact solution, particle fields
naturally \emph{adapt} to the flow field. It would thus also make sense to
adaptively regularize. The spline spaces discussed in this article famously
form a multi-resolution analysis and the approximate extension operator
yields approximations of smooth extensions on the whole-space. This opens
up the possibility to use wavelets. One way to achieve adaptive 
regularization might be to first choose $h=\sigma$ and compute the
regularized particle field as discussed in this paper. Afterwards one would
perform a fast wavelet transform on the regularized particle field and filter
out high-oscillatory components with large wavelet coefficients by a
thresholding procedure. Such an approach has been used successfully before
in the whole-space case~\cite{chehab2005} and might be able combine the best
of both approaches.

\section*{Acknowledgments}The author thanks Jan Giesselmann, Christian
Rieger, and Manuel Torrilhon for their comments on a preliminary version of
this manuscript. The author would furthermore like use the opportunity to
express his gratitude for his former PhD advisor Shinnosuke Obi of Keio
University for both his personal and scientific support. Many thanks also go to
the Japanese Ministry of Education (MEXT) for the scholarship support that
made the stay possible.

\appendix
\bibliographystyle{siamplain}
\bibliography{papers}
\end{document}